\documentclass[11pt]{amsart}
\usepackage{etex}
\usepackage{fullpage,pstricks,amsmath,amssymb,amsfonts,amsthm,verbatim,graphicx,fullpage,tikz,ifthen,mathdots}
\usepackage{hyperref}
\usepackage[all]{xy}

\newtheorem{theorem}{Theorem}[section]
\newtheorem{lemma}[theorem]{Lemma}
\newtheorem{proposition}[theorem]{Proposition}
\newtheorem{corollary}[theorem]{Corollary}
\newtheorem*{maincorollary}{Corollary 4.19}

\newtheorem*{maintheorem}{Theorem 3.5}
\newtheorem*{tamtheorem}{Theorem 4.11}
\newtheorem*{cattheorem}{Theorem 4.13}

\theoremstyle{definition}
\newtheorem{definition}[theorem]{Definition}
\newtheorem*{maindefinition}{Definition 2.12}
\newtheorem*{permdefinition}{Definition 3.1}

\newtheorem{remark}[theorem]{Remark}

\newcounter{x}
\newcounter{y}
\newcounter{z}

\newcommand\xaxis{210}
\newcommand\yaxis{-30}
\newcommand\zaxis{90}

\newcommand\topside[3]{
  \fill[fill=white, draw=black,shift={(\xaxis:#1)},shift={(\yaxis:#2)},
  shift={(\zaxis:#3)}] (0,0) -- (30:1) -- (0,1) --(150:1)--(0,0);
}

\newcommand\leftside[3]{
  \fill[fill=black, draw=black,shift={(\xaxis:#1)},shift={(\yaxis:#2)},
  shift={(\zaxis:#3)}] (0,0) -- (0,-1) -- (210:1) --(150:1)--(0,0);
}

\newcommand\rightside[3]{
  \fill[fill=gray, draw=black,shift={(\xaxis:#1)},shift={(\yaxis:#2)},
  shift={(\zaxis:#3)}] (0,0) -- (30:1) -- (-30:1) --(0,-1)--(0,0);
}

\newcommand\cube[3]{
  \topside{#1}{#2}{#3} \leftside{#1}{#2}{#3} \rightside{#1}{#2}{#3}
}

\newcommand\planepartition[1]{
 \setcounter{x}{-1}
  \foreach \a in {#1} {
    \addtocounter{x}{1}
    \setcounter{y}{-1}
    \foreach \b in \a {
      \addtocounter{y}{1}
      \setcounter{z}{-1}
      \foreach \c in {0,...,\b} {
        \addtocounter{z}{1}
      \ifthenelse{\c=0}{\setcounter{z}{-1},\addtocounter{y}{0}}{
        \cube{\value{x}}{\value{y}}{\value{z}}
      }
    }
  }
 }
}

\usetikzlibrary{arrows}
\usetikzlibrary{shapes}
\newcommand{\mymk}[1]{%
  \tikz[baseline=(char.base)]\node[anchor=south west, draw,rectangle, rounded corners, inner sep=2pt, minimum size=7mm,
    text height=2mm](char){\ensuremath{#1}} ;}

\begin{document}

%     [Short Title]{Full Title}
\title{Permutation totally symmetric \\ self-complementary plane partitions}

%    Information for  author
\author[J.\ Striker]{Jessica Striker}

\address{Department of Mathematics \\
North Dakota State University \\
Fargo, ND, USA}
\email{jessica.striker@ndsu.edu}

\thanks{The author is supported in part by the National Security Agency grant number H98230-15-1-0041, North Dakota EPSCoR grant number IIA-1355466, and the NDSU Advance FORWARD program sponsored by the National Science Foundation HRD-0811239.}

%    General info
%\subjclass{Primary 05A05; Secondary 06A07}

%\keywords{alternating sign matrix, permutation, plane partition, poset, Tamari lattice}
%\date{\today}

%\dedicatory{}
\begin{abstract}
Alternating sign matrices and totally symmetric self-complementary plane partitions are equinumerous sets of objects for which no explicit bijection is known. In this paper, we identify a subset of totally symmetric self-complementary plane partitions corresponding to permutations by giving a statistic-preserving bijection to permutation matrices, which are a subset of alternating sign matrices. We use this bijection to define a new partial order on permutations, and prove this new poset contains both the Tamari lattice and the Catalan distributive lattice as subposets. We also study a new partial order on totally symmetric self-complementary plane partitions arising from this perspective and show that this is a distributive lattice related to Bruhat order when restricted to permutations.
\end{abstract} 

\maketitle

%\tableofcontents

%%%%%%%%%%%%%%%%%%%%%%%%%%%%%%%%%%%%%%%%%%%%%%%%%%%%%%%%%%%%
\section{Introduction}
%%%%%%%%%%%%%%%%%%%%%%%%%%%%%%%%%%%%%%%%%%%%%%%%%%%%%%%%%%%%
\label{s:introduction}

Alternating sign matrices (ASM) with $n$ rows and $n$ columns, descending plane partitions (DPP) with largest part at most $n$, and totally symmetric self-complementary plane partitions (TSSCPP) inside a $2n\times 2n\times 2n$ box are equinumerous sets that lack an explicit bijection between any two. (See~\cite{ANDREWS_TSSCPP} \cite{KUP_ASM_CONJ} \cite{MRRASMDPP} \cite{MRRANDREWSCONJ} \cite{MRR3}   \cite{ZEILASM}  for these enumerations and bijective conjectures and \cite{BRESSOUDBOOK} for the story behind these papers.) In~\cite{Striker_DPP}, we gave a bijection between permutation matrices (which are a subset of ASM) and descending plane partitions with no special parts in such a way that the inversion number of the permutation matrix equals the number of parts of the DPP. In this paper, we complete the solution to this bijection problem in the special case of permutations by identifying a subset of TSSCPP corresponding to permutations and giving a bijection which yields a direct interpretation for the inversion number on these permutation TSSCPP. See Figure~\ref{fig:permbij}. We then use this bijection to prove several poset-theoretic results on TSSCPP and permutations, with connections to Bruhat order and partial orders on Catalan objects.

\begin{figure}[htbp]
\begin{center}
$\begin{array}{c}
\mbox{Totally Symmetric Self-}\\
\mbox{Complementary Plane Partition}\\
\scalebox{.2}{
$\begin{tikzpicture}
\planepartition{{12,12,12,12,12,12,10,10,10,10,6,6},{12,12,12,12,12,12,10,9,9,7,6,6},{12,12,11,11,11,10,8,8,6,6,5,2},{12,12,11,10,10,10,8,8,6,6,3,2},{12,12,11,10,8,8,6,6,4,4,3,2},{12,12,10,10,8,8,6,6,4,4,2,2},{10,10,8,8,6,6,4,4,2,2,0,0},{10,9,8,8,6,6,4,4,2,1,0,0},{10,9,6,6,4,4,2,2,2,1,0,0},{10,7,6,6,4,4,2,1,1,1,0,0},{6,6,5,3,3,2,0,0,0,0,0,0},{6,6,2,2,2,2,0,0,0,0,0,0}}
\end{tikzpicture}$}
\end{array}
\Leftrightarrow
\begin{array}{c}
\mbox{Permutation matrix}\\
\mbox{ }\\
\left( \begin{array}{cccccc}
0&0&0&1&0&0\\
0&0&0&0&0&1\\
0&0&1&0&0&0\\
0&0&0&0&1&0\\
1&0&0&0&0&0\\
0&1&0&0&0&0
\end{array}\right) \end{array}
%\Leftrightarrow
%\begin{array}{c}
%\mbox{Descending}\\
%\mbox{Plane Partition}\\
%\mbox{  }\\
%\begin{array}{ccccc}
%&&&&\\
%6 & 6 & 6 & 6 & 5\\
%& 5 & 4 & 4 & 4\\
%&  & 3 & 3 & \\
%&&&&\\
%&&&&
%\end{array} 
%\end{array}
$
\end{center}
\caption{A permutation matrix and the corresponding totally symmetric self-complementary plane partition} 
\label{fig:permbij}
\end{figure}

\medskip
\textbf{We highlight here the main new definitions and theorems of this paper.}
We define a new object which we show in Proposition~\ref{prop:boolbij} to be in bijection with TSSCPP. 
\begin{maindefinition}
A \emph{TSSCPP boolean triangle} of order $n$ is a triangular integer array $\{b_{i,j}\}$ for $1\leq i\leq n-1$, $n-i\leq j\leq n-1$ with entries in $\{0,1\}$ such that the diagonal partial sums satisfy 
%\begin{equation}
%\label{eq:par}
\[1+\displaystyle\sum_{i=j+1}^{i'} b_{i,n-j-1} \geq \displaystyle\sum_{i=j}^{i'} b_{i,n-j}.\]
%\end{equation}
\end{maindefinition}
\medskip

We use this new object to characterize the permutation subset of TSSCPP.

\begin{permdefinition}
Let \emph{permutation TSSCPP} of order $n$ be all TSSCPP inside a $2n\times 2n\times 2n$ box whose corresponding  boolean triangles have weakly decreasing rows. 
\end{permdefinition}

In Lemma~\ref{lem:permconfig} and Theorem~\ref{thm:permmagog}, we translate this characterization directly to TSSCPP and \emph{magog triangles}, which are triangular arrays in bijection with TSSCPP (see Definition~\ref{def:magog}).

\medskip
\textbf{Our main theorems} are as follows.

\begin{maintheorem}
There is a natural, statistic-preserving bijection between $n\times n$ permutation matrices with inversion number $p$  whose one in the last row is in column $k$ and whose one in the last column is in row $\ell$ and permutation TSSCPP boolean triangles of order~$n$ with $p$ zeros, 
exactly $n-k$ of which are contained in the last row, and for which the lowest one in diagonal $n-1$ is in row $\ell-1$. 
\end{maintheorem}

We use this characterization of permutation TSSCPP to define, in Definition~\ref{def:permmagog}, $T_n^{\rm Perm}$, \textbf{a new partial order on permutations}, as a restriction of the natural partial order on magog triangles. We show that $T_n^{\rm Perm}$ contains two non-isomorphic Catalan subposets.

\begin{tamtheorem}
$T_n^{\rm Perm}$ contains a subposet which is isomorphic to the \textbf{Tamari lattice} $Tam_n$. In particular, the subposet of $T_n^{\rm Perm}$ consisting of the $132$-avoiding permutations is isomorphic to $Tam_n$.
\end{tamtheorem}

\begin{cattheorem}
$T_n^{\rm Perm}$ contains a subposet which is isomorphic to the \textbf{Catalan distributive lattice} $Cat_n$. In particular, the subposet of $T_n^{\rm Perm}$ consisting of the $213$-avoiding permutations is isomorphic to $Cat_n$.
\end{cattheorem}
We also define \textbf{a new poset on TSSCPP}, $TBool_n$, using TSSCPP boolean triangles. Of note is the following corollary of Theorem~\ref{thm:tsscppsn}. 

\begin{maincorollary}
$TBool_n^{\rm Perm}$ (the induced subposet of $TBool_n$ consisting of all the permutation TSSCPP boolean triangles) equals $[2]\times[3]\times\cdots\times[n]$, that is, the product of chains with $2, 3, \ldots, n$ elements. Thus, this
 is a partial order on permutations which sits between the weak and strong \textbf{Bruhat orders} on the symmetric group. 
\end{maincorollary}

\medskip
\textbf{The outline is as follows.}
In Section~\ref{sec:nilpmt}, we define totally symmetric self-complementary plane partitions and alternating sign matrices and give bijections within their respective families. In Section~\ref{subsec:asmfam}, we recall the standard bijection from an ASM to a \emph{monotone triangle}. We give, in Section~\ref{subsec:tsscppfam}, known bijections from a TSSCPP to a \emph{magog triangle} and a certain \emph{nest of non-intersecting lattice paths}. We transform this last object to a new object which we call a \emph{boolean triangle}.
 
In Section~\ref{sec:conseq}, we \emph{identify the permutation TSSCPP subset} in terms of the boolean triangles of Section~\ref{subsec:tsscppfam}, and then translate this condition to the other objects in the TSSCPP family, including magog triangles. In Section~\ref{subsec:bij}, we use this characterization to present in Theorem~\ref{thm:tsscppsn} a \emph{direct, statistic-preserving bijection} between this TSSCPP subset and permutation matrices.
In Section~\ref{sec:spec}, we discuss the outlook of the general ASM--TSSCPP bijection problem and compare the bijection of this paper  
with other bijections between subsets of ASM and TSSCPP.

In Section~\ref{sec:poset}, we contrast various poset structures on ASM and TSSCPP as well as their permutation and Catalan subposets.
In Section~\ref{subsec:AnTn}, we review known \emph{distributive lattices} on ASM and TSSCPP. In Section~\ref{subsec:tamari}, we define and study \emph{a new partial order on permutations} which appears as the permutation subposet of the TSSCPP distributive lattice from Section~\ref{subsec:AnTn}. We prove Theorems~\ref{thm:tamari} and \ref{thm:Catdistr}, identifying \emph{two Catalan subposets} of this permutation order.
In Section~\ref{subsec:TBool}, we study \emph{a new partial order on TSSCPP} obtained via boolean triangles. We show in Corollary~\ref{cor:TBool} that the subposet of permutation boolean triangles is an especially nice distributive lattice that sits between the weak and strong Bruhat orders. In Section~\ref{subsec:summary}, we summarize the results of Section~\ref{sec:poset}. 

\section{The objects and their alter egos: ASM / monotone triangle, TSSCPP / non-intersecting lattice paths / boolean triangle}
\label{sec:nilpmt}

In this section, we first define alternating sign matrices and recall the standard bijection to monotone triangles. We then define totally symmetric self-complementary plane partitions and give bijections with non-intersecting lattice paths and new objects we call boolean triangles.
Then in Section~\ref{s:sn}, we give a bijection in the permutation case via these intermediary objects.

\subsection{The ASM family}
\label{subsec:asmfam}
\begin{definition}[\cite{MRRANDREWSCONJ}]
An \emph{alternating sign matrix (ASM)} is a square matrix with entries in $\{0, 1, -1\}$ whose rows and columns each sum to one and such that the nonzero entries along each row or column alternate in sign.
\end{definition}

\begin{figure}[htbp]
\scalebox{.9}{
$
\left( 
\begin{array}{rrr}
1 & 0 & 0 \\
0 & 1 & 0\\
0 & 0 & 1
\end{array} \right)
\left( 
\begin{array}{rrr}
1 & 0 & 0 \\
0 & 0 & 1\\
0 & 1 & 0
\end{array} \right)
\left( 
\begin{array}{rrr}
0 & 1 & 0 \\
1 & 0 & 0\\
0 & 0 & 1
\end{array} \right)
\left( 
\begin{array}{rrr}
0 & 1 & 0 \\
1 & -1 & 1\\
0 & 1 & 0
\end{array} \right)$}

\vspace{.1in}
\scalebox{.9}{$
\left( 
\begin{array}{rrr}
0 & 1 & 0 \\
0 & 0 & 1\\
1 & 0 & 0
\end{array} \right)
\left( 
\begin{array}{rrr}
0 & 0 & 1 \\
1 & 0 & 0\\
0 & 1 & 0
\end{array} \right)
\left( 
\begin{array}{rrr}
0 & 0 & 1 \\
0 & 1 & 0\\
1 & 0 & 0
\end{array} \right)
$
}
\caption{The seven $3\times 3$ ASM}
\label{fig:3x3asms}
\end{figure}

See Figure~\ref{fig:3x3asms} for the seven $3\times 3$ ASM. It is clear that the alternating sign matrices with no negative  entries are the permutation matrices.

Alternating sign matrices are well known to be in bijection with monotone triangles, which are \emph{strict Gelfand-Tsetlin patterns} with \emph{highest weight} $1~2~3~\cdots~n$. A direct definition is below.
\begin{definition}[\cite{MRRASMDPP}]
\label{def:mt}
A \emph{monotone triangle} of order $n$
is a triangular array of 
%$i$~integers in row $i$ for all $1\le i\le n$, 
integers $a_{i,j}$ for $1\leq i\leq n$, $n-i\leq j\leq n-1$ with bottom row entries $a_{n,j}=j+1$ and all other $a_{i,j}$ satisfying $a_{i+1,j-1} \le a_{i,j} \le a_{i+1,j} \mbox{ and } a_{i,j} < a_{i,j+1}$.
\end{definition}

\begin{figure}[htbp]
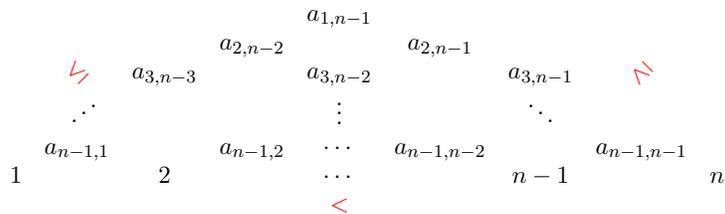

\begin{center}
\scalebox{.85}{
$\begin{array}{ccccccccccc}
  & & & & & a_{1,n-1} & & & & & \\
  & & & & a_{2,n-2} & & a_{2,n-1} & & & & \\
  & &\textcolor{red}{\rotatebox[origin=c]{45}{$\mathbf{\leq}$}} & a_{3,n-3} & & a_{3,n-2} & & a_{3,n-1} & \textcolor{red}{\rotatebox[origin=c]{135}{$\mathbf{\geq}$}} & & \\
  &  & \mbox{ }\iddots &  & & \vdots & &   \ddots & &  &\\
  &  & a_{n-1,1} & & a_{n-1,2} & \cdots & a_{n-1,n-2} & & a_{n-1,n-1} & &\\
  & 1 & & 2 & & \cdots & &  n-1 & & n & \\
  & & & & & \textcolor{red}{\mathbf{<}} & & & & &
\end{array}$}
\end{center}
\caption{A generic monotone triangle}
\label{fig:montri}
\end{figure}

\begin{figure}[htbp]
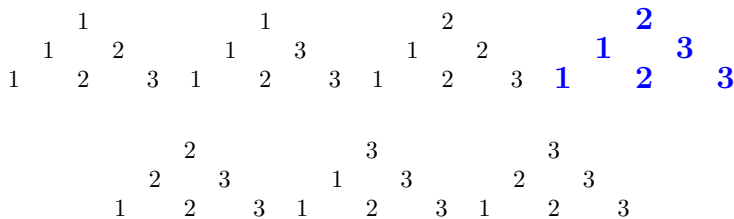

\begin{center}
\scalebox{.85}{
$\begin{array}{ccccc}
&& 1 &\\
&1 & & 2\\
1&&2&&3
\end{array}$
%\hspace{.2in}
$\begin{array}{ccccc}
&& 1 &\\
&1 & &3\\
1&&2&&3
\end{array}$
%\hspace{.2in}
$\begin{array}{ccccc}
&& 2 &\\
&1 & &2\\
1&&2&&3
\end{array}$
%\hspace{.2in}
%\hspace{.2in}
$\begin{array}{ccccc}
&& \Large{\textcolor{blue}{\textbf{2}}} &\\
&\Large{\textcolor{blue}{\textbf{1}}} & &\Large{\textcolor{blue}{\textbf{3}}}\\
\Large{\textcolor{blue}{\textbf{1}}}&&\Large{\textcolor{blue}{\textbf{2}}}&&\Large{\textcolor{blue}{\textbf{3}}}
\end{array}$}
\vspace{.2in}

\scalebox{.85}{
$\begin{array}{ccccc}
&& 2 &\\
&2 & &3\\
1&&2&&3
\end{array}$
%\hspace{.2in}
$\begin{array}{ccccc}
&& 3 &\\
&1 & &3\\
1&&2&&3
\end{array}$
%\hspace{.2in}
$\begin{array}{ccccc}
&& 3 &\\
&2 & &3\\
1&&2&&3
\end{array}$
}
\end{center}
\caption{The seven monotone triangles of order three, listed in the order corresponding to Figure~\ref{fig:3x3asms}.  The non-permutation is shown in bold.}
\label{fig:mt3}
\end{figure}

\begin{proposition}[\cite{MRRASMDPP}]
Monotone triangles of order $n$ are in bijection with $n\times n$ ASM.
\end{proposition}

\begin{proof}
To create row $i$ of a monotone triangle $a$, in row $i$ of the corresponding ASM $A$, note which columns have a partial sum (from the top) of 1 in that row. Record the numbers of the columns in which this occurs in increasing order, thus $a_{i,j} < a_{i,j+1}$ by construction. Since each column of $A$ sums to one, the last row of the monotone triangle will record all column numbers $1~2~\cdots n$. Finally, by the alternating condition of the ASM, the monotone triangle entries will satisfy the diagonal inequalities $a_{i+1,j-1} \le a_{i,j} \le a_{i+1,j}$. This process is clearly invertible and so is a bijection.
\end{proof}

See Figures~\ref{fig:montri} and~\ref{fig:mt3}. This bijection yields the following easy corollary.
 \begin{corollary}[\cite{MRRASMDPP}]
 \label{cor:negativeone}
Monotone triangle entries $a_{i,j}$ satisfying the strict diagonal inequalities $a_{i,j-1} < a_{i-1,j} < a_{i,j}$ are associated with the $-1$ entries of the corresponding ASM. 
\end{corollary}

We will need the notion of \emph{inversion number} in an alternating sign matrix. First, recall that in a permutation $i\rightarrow \sigma(i)$, the \emph{inversion number} is the number of pairs $(i,j)$ such that $i<j$ and $\sigma(j)<\sigma(i)$. We now recall the following definition of the inversion number of an alternating sign matrix; this definition extends the usual notion of inversion in a permutation. 

\begin{definition}\rm
\label{def:inv}(\cite{MRRASMDPP})
The \emph{inversion number} of an ASM $A$ is defined as $I(A)=\sum A_{ij} A_{k\ell}$ where the sum is over all $i,j,k,\ell$ such that $i>k$ and $j<\ell$.  
\end{definition}

Note there is some ambiguity in the literature regarding the definition of the inversion number of an ASM. We use the above definition found in~\cite{MRRASMDPP} and~\cite{BRESSOUDBOOK}; an alternative definition is this number minus the number of negative ones in the ASM, and is the definition introduced in~\cite{RobbinsRumsey} under the name \emph{positive inversions}. See p.\ 443 of~\cite{BehrendMultASM} for discussion. Note that in the permutation case, these two notions of inversion number are equivalent, so for this paper, the choice of convention is irrelevant.

There are many other objects in bijection with alternating sign matrices, such as height function matrices, fully-packed loop configurations, and square ice configurations; see~\cite{ProppManyFaces}.

\subsection{The TSSCPP family}
\label{subsec:tsscppfam}
There are a few different ways to define plane partitions, including as two-dimensional decreasing integer arrays and as three-dimensional stacks of cubes in a corner. We will need both of these perspectives, and so define both below.

\begin{definition}
\label{def:tsscpp}
A \emph{plane partition} is a two dimensional array of positive integers $\{t_{i,j}\}$ that weakly decreases across rows from left to right and down columns. We will sometimes take the array to be rectangular, in which case we complete the array to a rectangle by adding entries equal to zero. So we have the following $n\times m$ array satisfying $t_{i,j}\geq t_{i+1,j}\geq 0$ and $t_{i,j}\geq t_{i,j+1}\geq 0$.

\vspace{.1in}
\begin{center}
$\begin{array}{ccccc}
 t_{1,1} & t_{1,2} & t_{1,3} & \cdots & t_{1,m} \\
 t_{2,1} & t_{2,2} & t_{2,3} &  & t_{2,m} \\
 \vdots & & & & \vdots \\
 t_{n,1} & t_{n,2} & t_{n,3} & \cdots & t_{n,m} \\
\end{array}$
\end{center}
\end{definition}

\begin{figure}[htbp]
\begin{center}
\scalebox{.7}{$
\begin{array}{cccccc}
6&6&6&3&3&3\\
6&6&6&3&3&3\\
6&6&6&3&3&3\\
3&3&3&&&\\
3&3&3&&&\\
3&3&3&&&
\end{array} 
\hspace{.7cm}
\begin{array}{cccccc}
6&6&6&4&3&3\\
6&6&6&3&3&3\\
6&6&5&3&3&2\\
4&3&3&\Large{\textbf{\textcolor{red}{1}}}&&\\
3&3&3&&&\\
3&3&2&&&
\end{array} 
\hspace{.7cm}
\begin{array}{cccccc}
6&6&6&5&4&3\\
6&6&5&3&3&2\\
6&5&5&3&3&1\\
5&3&3&\Large{\textcolor{red}{\textbf{1}}}&\Large{\textbf{\textcolor{red}{1}}}&\\
4&3&3&1&&\\
3&2&1&&&
\end{array} 
\hspace{.7cm}
\begin{array}{cccccc}
6&6&6&5&4&3\\
6&6&5&4&3&2\\
6&5&4&3&2&1\\
5&4&3&\Large{\textbf{\textcolor{red}{2}}}&\Large{\textbf{\textcolor{red}{1}}}&\\
4&3&2&1&&\\
3&2&1&&&
\end{array} 
$}

\vspace{.1in}
\scalebox{.7}{$
\begin{array}{cccccc}
6&6&6&4&3&3\\
6&6&6&4&3&3\\
6&6&4&3&2&2\\
4&4&3&\Large{\textbf{\textcolor{red}{2}}}&&\\
3&3&2&&&\\
3&3&2&&&
\end{array} 
\hspace{.7cm}
\begin{array}{cccccc}
6&6&6&5&5&3\\
6&5&5&3&3&1\\
6&5&5&3&3&1\\
5&3&3&\Large{\textbf{\textcolor{red}{1}}}&\Large{\textbf{\textcolor{red}{1}}}&\\
5&3&3&1&\Large{\textbf{\textcolor{red}{1}}}&\\
3&1&1&&&
\end{array} 
\hspace{.7cm}
\begin{array}{cccccc}
6&6&6&5&5&3\\
6&5&5&4&3&1\\
6&5&4&3&2&1\\
5&4&3&\Large{\textbf{\textcolor{red}{2}}}&\Large{\textcolor{red}{\textbf{1}}}&\\
5&3&2&1&\Large{\textbf{\textcolor{red}{1}}}&\\
3&1&1&&&
\end{array} 
$}
\end{center}
\caption{TSSCPP inside a $6\times 6\times 6$ box as integer arrays, as in Definition~\ref{def:tsscpp}. The fundamental domain entries are shown in a larger bold font.}
\label{fig:cl10}
\end{figure}

\medskip
We can visualize a plane partition as a stack of unit cubes pushed into the corner of a room. If we identify the corner of the room with the origin and the room with the positive orthant and denote each unit cube by the coordinate of its corner farthest from the origin in $\mathbb{N}^3$, we obtain the following equivalent definition, by which we may also define various symmetry classes.

\begin{definition}
\label{def:cubes}
 A \emph{plane partition} $\pi$ is a finite set of positive integer lattice points $(i,j,k)$ such that if $(i,j,k)\in\pi$ and $1\le i'\le i$, $1\le j'\le j$, and $1\le k'\le k$ then $(i',j',k')\in\pi$.

A plane partition $\pi$ is \emph{symmetric} if whenever $(i,j,k)\in \pi$ then $(j,i,k)\in \pi$ as well. $\pi$ is \emph{cyclically symmetric} if whenever $(i,j,k)\in\pi$ then $(j,k,i)\in\pi$ as well. 
A plane partition is \emph{totally symmetric} if 
it is both symmetric and cyclically symmetric, so that  
whenever $(i,j,k)\in\pi$ then all six permutations of $(i,j,k)$ are also in $\pi$.

 A plane partition is \emph{self-complementary} inside a given bounding box $a\times b\times c$ if it is equal to its complement in the box, that is, the collection of empty cubes in the box is of the same shape as the collection of cubes in the plane partition itself.  
A \emph{totally symmetric self-complementary plane partition} (TSSCPP) inside a 
$2n \times 2n \times 2n$ box is a plane partition which is both totally symmetric and self-complementary inside the box.
\end{definition}

\begin{figure}[htbp]
\begin{center}
$\begin{gathered}\scalebox{0.225}{
$\begin{tikzpicture}
\planepartition{{6,6,6,3,3,3},{6,6,6,3,3,3},{6,6,6,3,3,3},{3,3,3},{3,3,3},{3,3,3}}
\end{tikzpicture}
\hspace{.7cm}
\begin{tikzpicture}
\planepartition{{6,6,6,4,3,3},{6,6,6,3,3,3},{6,6,5,3,3,2},{4,3,3,1},{3,3,3},{3,3,2}}
\end{tikzpicture}
\hspace{.7cm}
\begin{tikzpicture}
\planepartition{{6,6,6,5,4,3},{6,6,5,3,3,2},{6,5,5,3,3,1},{5,3,3,1,1},{4,3,3,1},{3,2,1}}
\end{tikzpicture}
\hspace{.7cm}
\begin{tikzpicture}
\planepartition{{6,6,6,5,4,3},{6,6,5,4,3,2},{6,5,4,3,2,1},{5,4,3,2,1},{4,3,2,1},{3,2,1}}
\end{tikzpicture}$
}\end{gathered}$

$\begin{gathered}\scalebox{0.225}{
\begin{tikzpicture}
\planepartition{{6,6,6,4,3,3},{6,6,6,4,3,3},{6,6,4,3,2,2},{4,4,3,2},{3,3,2},{3,3,2}}
\end{tikzpicture}
\hspace{.7cm}
\begin{tikzpicture}
\planepartition{{6,6,6,5,5,3},{6,5,5,3,3,1},{6,5,5,3,3,1},{5,3,3,1,1},{5,3,3,1,1},{3,1,1}}
\end{tikzpicture}
\hspace{.7cm}
\begin{tikzpicture}
\planepartition{{6,6,6,5,5,3},{6,5,5,4,3,1},{6,5,4,3,2,1},{5,4,3,2,1},{5,3,2,1,1},{3,1,1}}
\end{tikzpicture}
}\end{gathered}.$
\end{center}
\caption{TSSCPP inside a $6\times 6\times 6$ box drawn as stacks of cubes in a corner, as in Definition~\ref{def:cubes}}
\label{ex:tikztsscpp}
\end{figure}
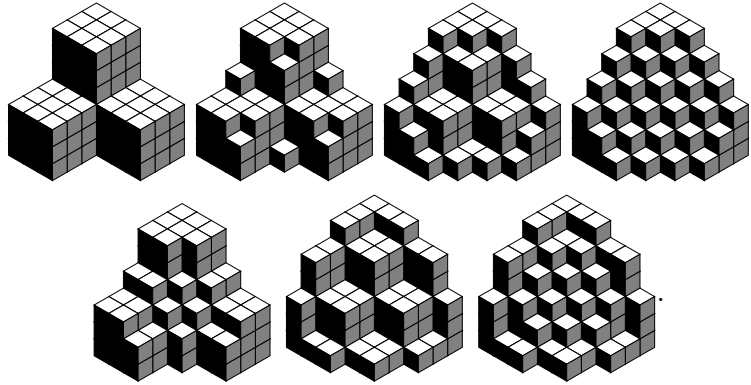

See Figures~\ref{fig:cl10} and \ref{ex:tikztsscpp} for the seven TSSCPP inside a $6\times 6\times 6$ box. 

\medskip
Because of the large amount of symmetry in a TSSCPP, we can record the defining information in a more compact form. That is, we take a \emph{fundamental domain}, which amounts to the triangular region shown in Figure~\ref{ex:funddomain3} and, more formally, the entries in Definition~\ref{def:funddomain} below. See also Figure~\ref{fig:cl10}.
\begin{definition}
\label{def:funddomain}
Considering a TSSCPP as the array of integers below with weakly decreasing rows and columns
\begin{center}
$\begin{array}{ccccc}
 t_{1,1} & t_{1,2} & t_{1,3} & \cdots & t_{1,2n} \\
 t_{2,1} & t_{2,2} & t_{2,3} &  & t_{2,2n} \\
 \vdots & & & & \vdots \\
 t_{2n,1} & t_{2n,2} & t_{2n,3} & \cdots  & t_{2n,2n} \\
\end{array}$
\end{center}
we define the \emph{fundamental domain} of a TSSCPP as the following entries:
\begin{center}
$\begin{array}{ccccc}
 t_{n+1,n+1} & t_{n+1,n+2} & t_{n+1,n+3} & \cdots & t_{n+1,2n} \\
  & t_{n+2,n+2} & t_{n+2,n+3} &  & t_{n+2,2n} \\
  & & \ddots & & \vdots \\
  & & & & t_{2n,2n}. \\
\end{array}$
\end{center}
\end{definition}

\medskip
We can record the information from the fundamental domain in either of two ways, namely, by counting the number of boxes in each stack or drawing the lattice paths traced out at each level. Under a slight deformation, the first way corresponds to the \emph{magog triangles} of Definition~\ref{def:magog} below. The second way corresponds to certain \emph{nests of non-intersecting lattice paths} described in Proposition~\ref{prop:nilp}.

\begin{figure}
\begin{center}
\includegraphics[scale=.56]{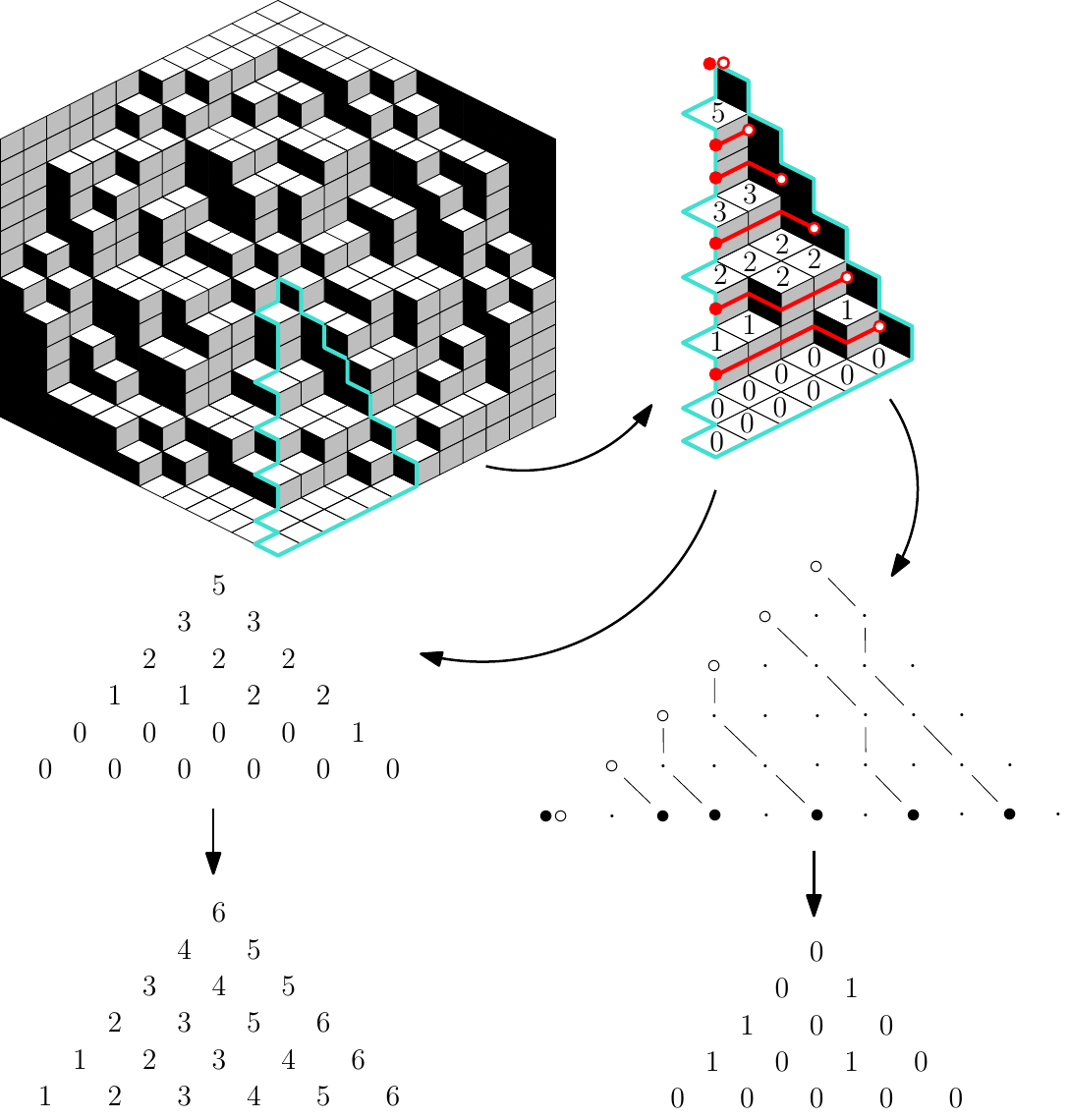}
\end{center}
\caption{A TSSCPP (upper left) and its fundamental domain (upper right), with its corresponding magog triangle (lower left), nest of non-intersecting lattice paths (middle right), and boolean triangle (lower right)}
\label{ex:funddomain3}
\end{figure}

\begin{definition}
\label{def:magog}
A \emph{magog triangle} of order $n$
is a triangular array of integers $\alpha_{i,j}$ for $1\leq i\leq n$, $n-i\leq j\leq n-1$, such that the bottom row entries are given by $\alpha_{n,j}=j+1$ and all other $\alpha_{i,j}$ satisfy $\alpha_{i,j}\leq n$,  $\alpha_{i,j} +1\geq \alpha_{i+1,j} \mbox{ and } \alpha_{i,j} < \alpha_{i,j+1}$ (note these conditions also imply $\alpha_{i+1,j-1} \le \alpha_{i,j}$).
\end{definition}

In~\cite{MRR3}, W.~Mills, D.~Robbins, and H.~Rumsey give a bijection between the magog triangles of the above definition and TSSCPP. In D.~Zeilberger's proof of the enumeration of alternating sign matrices~\cite{ZEILASM}, he gave these triangles the name \emph{magog triangles} (and, likewise, called monotone triangles \emph{gog triangles}). See Figures~\ref{fig:magogtri} and \ref{fig:magog3}.

\begin{figure}[htbp]
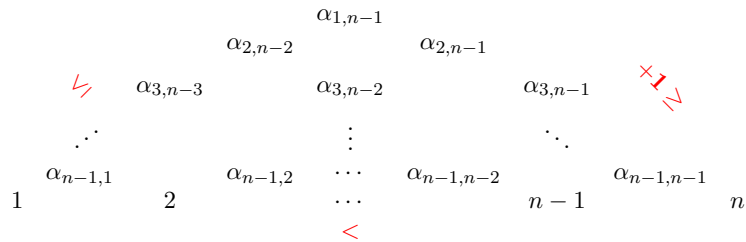

\begin{center}
\scalebox{.85}{
$\begin{array}{ccccccccccc}
  & & & & & \alpha_{1,n-1} & & & & & \\
  & & & & \alpha_{2,n-2} & & \alpha_{2,n-1} & & & & \\
  & &\textcolor{red}{\rotatebox[origin=c]{45}{$\mathbf{\leq}$}} & \alpha_{3,n-3} & & \alpha_{3,n-2} & & \alpha_{3,n-1} & \textcolor{red}{\rotatebox[origin=c]{315}{$\mathbf{ + 1 \geq}$}} & & \\
  &  & \mbox{ }\iddots &  & & \vdots & &   \ddots & &  &\\
  &  & \alpha_{n-1,1} & & \alpha_{n-1,2} & \cdots & \alpha_{n-1,n-2} & & \alpha_{n-1,n-1} & &\\
  & 1 & & 2 & & \cdots & &  n-1 & & n & \\
  & & & & & \textcolor{red}{\mathbf{<}} & & & & &
\end{array}$}
\end{center}
\caption{A generic magog triangle; the ``\textcolor{red}{$+1\geq$}'' symbol represents the inequality $\alpha_{i,j} +1\geq \alpha_{i+1,j}$}
\label{fig:magogtri}
\end{figure}

\begin{proposition}[\cite{MRR3}]
\label{prop:magog}
Magog triangles of order $n$ are in bijection with TSSCPP inside a $2n\times 2n\times 2n$ box.
\end{proposition}
\begin{proof}
The magog triangle $\alpha$ corresponding to a TSSCPP $t$ is defined by rotating the fundamental domain and then adding $1,2,3,\ldots, n$ to the diagonals. That is, $\alpha_{i,n-j}=t_{n+j,n+i}+i-j+1$. See Figure~\ref{ex:funddomain3}.
\end{proof}

\begin{figure}[htbp]
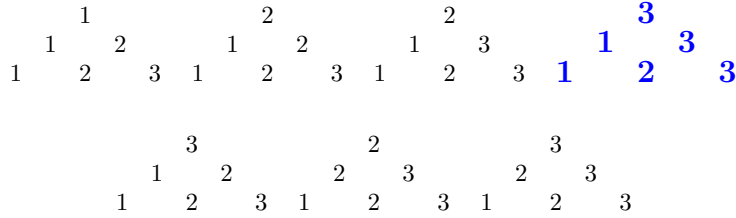

\begin{center}
\scalebox{.85}{
$\begin{array}{ccccc}
&& 1 &\\
&1 & & 2\\
1&&2&&3
\end{array}$
%\hspace{.2in}
$\begin{array}{ccccc}
&& 2 &\\
&1 & &2\\
1&&2&&3
\end{array}$
%\hspace{.2in}
$\begin{array}{ccccc}
&& 2 &\\
&1 & &3\\
1&&2&&3
\end{array}$
%\hspace{.2in}
$\begin{array}{ccccc}
&& \Large{\textcolor{blue}{\textbf{3}}} &\\
&\Large{\textcolor{blue}{\textbf{1}}} & &\Large{\textcolor{blue}{\textbf{3}}}\\
\Large{\textcolor{blue}{\textbf{1}}}&&\Large{\textcolor{blue}{\textbf{2}}}&&\Large{\textcolor{blue}{\textbf{3}}}
\end{array}$}
\vspace{.2in}

\scalebox{.85}{
$\begin{array}{ccccc}
&& 3 &\\
&1 & &2\\
1&&2&&3
\end{array}$
%\hspace{.2in}
$\begin{array}{ccccc}
&& 2 &\\
&2 & &3\\
1&&2&&3
\end{array}$
%\hspace{.2in}
$\begin{array}{ccccc}
&& 3 &\\
&2 & & 3\\
1&&2&&3
\end{array}$
}
\end{center}
\caption{The seven magog triangles of order three, listed in the order corresponding to Figures~\ref{fig:cl10}, \ref{ex:tikztsscpp}, and, by the permutation case bijection of Theorem~\ref{thm:tsscppsn}, Figures~\ref{fig:3x3asms} and \ref{fig:mt3}. The non-permutation is shown in bold.}
\label{fig:magog3}
\end{figure}

In~\cite{Doran}, W.~Doran gave a bijection from TSSCPP inside a $2n\times 2n\times 2n$ box to certain nests of non-intersecting lattice paths. In~\cite{DiFrancesco_NILP}, P.~Di Francesco gave a bijection to an equivalent collection of non-intersecting lattice paths. Either bijection proceeds by taking a fundamental domain of the TSSCPP, and instead of reading the number of boxes in each stack, one draws the paths going alongside those boxes. This yields a collection of non-intersecting paths with two types of steps. With a slight further deformation, namely, a quarter-turn counterclockwise, we obtain that the following objects are in bijection with TSSCPP. This is equivalent to vertically reflecting the paths in \cite{DiFrancesco_NILP}. See Figures~\ref{ex:funddomain3} and \ref{fig:NILPn3}.

\begin{proposition}[\cite{Doran}, \cite{DiFrancesco_NILP}]
\label{prop:nilp}
Totally symmetric self-complementary plane partitions inside a $2n \times 2n \times 2n$ box are in bijection with nests of non-intersecting lattice paths starting at $(i,i)$, $i=1,2,\ldots,n-1$, and ending at positive integer points on the $x$-axis of the form $(r_i, 0)$, $i = 1, 2,\ldots, n-1$, making only vertical down steps $(0, -1)$ or southeast diagonal steps $(1, -1)$.
\end{proposition}

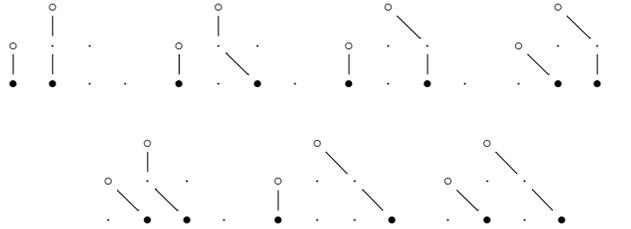
\begin{figure}[htbp]
\raisebox{\depth}{\scalebox{1}[-1]{
\scalebox{0.62}{
$\xymatrix @-1.0pc {
 \bullet \ar@{-}[d] & \bullet \ar@{-}[d] & \cdot & \cdot  \\
 \circ & \cdot \ar@{-}[d] & \cdot  \\
  & \circ &
}$
$\xymatrix @-1.0pc {
 & \bullet \ar@{-}[d] & \cdot & \bullet \ar@{-}[dl]  & \cdot  \\
 & \circ & \cdot \ar@{-}[d] & \cdot  \\
 &  & \circ &
}$
$\xymatrix @-1.0pc {
 & \bullet \ar@{-}[d] & \cdot & \bullet \ar@{-}[d] & \cdot  \\
 & \circ & \cdot & \cdot \ar@{-}[dl]  \\
 &  & \circ &
}$
$\xymatrix @-1.0pc {
 & \cdot & \bullet \ar@{-}[dl] & \bullet \ar@{-}[d] & \cdot \\
 & \circ & \cdot & \cdot \ar@{-}[dl] \\
 &  & \circ &
}$}}}
\vspace{.2in}

\raisebox{\depth}{\scalebox{1}[-1]{
\scalebox{.62}{
$\xymatrix @-1.0pc {
 & \cdot & \bullet \ar@{-}[dl] & \bullet \ar@{-}[dl] & \cdot \\
 & \circ & \cdot \ar@{-}[d] & \cdot \\
 &  & \circ &
}$
$\xymatrix @-1.0pc {
 & \bullet \ar@{-}[d] & \cdot & \cdot & \bullet \ar@{-}[dl] \\
 & \circ & \cdot & \cdot \ar@{-}[dl] \\
 &  & \circ &
}$
$\xymatrix @-1.0pc {
 & \cdot & \bullet \ar@{-}[dl] & \cdot & \bullet \ar@{-}[dl] \\
 & \circ & \cdot & \cdot \ar@{-}[dl] \\
 &  & \circ &
}$
}
}}
\caption{The seven TSSCPP nests of non-intersecting lattice paths of order three, in the order corresponding to Figure~\ref{fig:magog3}}
\label{fig:NILPn3}
\end{figure}

In~\cite{DiFrancesco_NILP}, Di Francesco uses the Lindstr\"{o}m-Gessel-Viennot formula for counting non-intersecting lattice paths via a determinant evaluation to give an expression for the TSSCPP generating function with a weight of $\tau$ per vertical step.  
We will show that when restricted to permutation TSSCPP, the diagonal steps correspond to the inversions of the permutation, so the vertical steps correspond to non-inversions.
Note that a specialization of this $\tau$-weighted TSSCPP generating function equals the degree of the variety of strictly upper triangular $2n\times 2n$ complex matrices of vanishing square~\cite{difran_zinn06}. See also~\cite{difran_zinn08} and Section 4.2 of \cite{integrable}.

We replace up steps in the nest of non-intersecting lattice paths by ones and diagonal steps by zeros to obtain a new triangular integer array we call a \emph{TSSCPP boolean triangle}, which we will use in our main bijection in Section~\ref{subsec:bij}. See Figures~\ref{fig:booltri} and~\ref{fig:booln3}.

\begin{definition}
\label{def:bool}
A \emph{TSSCPP boolean triangle} of order $n$ is a triangular integer array $\{b_{i,j}\}$ for $1\leq i\leq n-1$, $n-i\leq j\leq n-1$ with entries in $\{0,1\}$ such that the diagonal partial sums satisfy 
%\begin{equation}
%\label{eq:par}
\[1+\displaystyle\sum_{i=j+1}^{i'} b_{i,n-j-1} \geq \displaystyle\sum_{i=j}^{i'} b_{i,n-j}.\]
%\end{equation}
\end{definition}

\medskip

\begin{proposition}
\label{prop:boolbij}
TSSCPP boolean triangles of order $n$ are in bijection with TSSCPP inside a 
$2n \times 2n \times 2n$ box.
\end{proposition}

\begin{proof}
The bijection proceeds by replacing each vertical step in the nest of non-intersecting lattice paths with a one and each diagonal step with a zero. The inequality on the partial sums says that the partial sum of any northwest to southeast diagonal is not more than one larger than the partial sum of the diagonal to its left. This is equivalent to the condition that the lattice paths are non-intersecting.
\end{proof}

\begin{figure}[htbp]
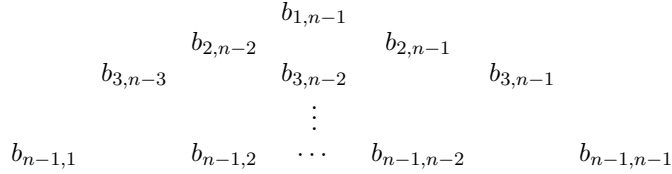

\begin{center}
\scalebox{.9}{
$\begin{array}{ccccccccccc}
  & & & & & b_{1,n-1} & & & & & \\
  & & & & b_{2,n-2} & & b_{2,n-1} & & & & \\
  & & & b_{3,n-3} & & b_{3,n-2} & & b_{3,n-1} & & & \\
  &  & &  & & \vdots & &  & &  &\\
  &  & b_{n-1,1} & & b_{n-1,2} & \cdots & b_{n-1,n-2} & & b_{n-1,n-1} & &
\end{array}$}
\end{center}
\caption{Indexing convention for entries in a generic TSSCPP boolean triangle}
\label{fig:booltri}
\end{figure}

See Figure~\ref{fig:booln3} for the seven TSSCPP boolean triangles of order three. Note that \scalebox{.8}{$\begin{array}{ccc}
& 1 &\\
0 & &1
\end{array}$}
violates the defining inequality condition and is thus not a TSSCPP boolean triangle.

\begin{figure}[htbp]
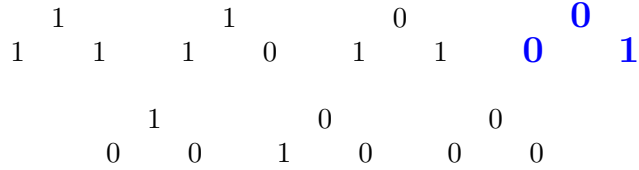

\begin{center}
$\begin{array}{ccc}
& 1 &\\
1 & & 1
\end{array}$
\hspace{.15in}
$\begin{array}{ccc}
& 1 &\\
1 & &0
\end{array}$
\hspace{.15in}
$\begin{array}{ccc}
& 0 &\\
1 & &1
\end{array}$
\hspace{.15in}
$\begin{array}{ccc}
& \Large{\textbf{\textcolor{blue}{0}}} &\\
\Large{\textbf{\textcolor{blue}{0}}} & &\Large{\textbf{\textcolor{blue}{1}}}
\end{array}$

\vspace{.15in}
$\begin{array}{ccc}
& 1 &\\
0 & &0
\end{array}$
\hspace{.15in}
$\begin{array}{ccc}
& 0 &\\
1 & &0
\end{array}$
\hspace{.15in}
$\begin{array}{ccc}
& 0 &\\
0 & &0
\end{array}$
\end{center}
\caption{The seven TSSCPP boolean triangles of order three, listed in the order corresponding to Figures~\ref{fig:cl10}, \ref{fig:magog3}, and~\ref{fig:NILPn3}, and Figure~\ref{fig:mt3} via the permutation case bijection of Theorem~\ref{thm:tsscppsn}. The non-permutation is shown in bold.}
\label{fig:booln3}
\end{figure}

%%%%%%%%%%%%%%%%%%%%%%%%%%%%%%%%%%%%%%%%%%%%%%%%%%%%%%%%%%%%
\section{A bijection on permutations}
%%%%%%%%%%%%%%%%%%%%%%%%%%%%%%%%%%%%%%%%%%%%%%%%%%%%%%%%%%%%
\label{s:sn}

In this section, we prove our first main result, Theorem~\ref{thm:tsscppsn}, giving a bijection between $n\times n$ permutation matrices and a subset of totally symmetric self-complementary plane partitions inside a $2n\times 2n\times 2n$ box which preserves the inversion number statistic and two boundary statistics. 
First, in Definition~\ref{def:permbool} and  Proposition~\ref{prop:nfact} we identify and enumerate the permutation subset of TSSCPP boolean triangles; we then translate to the other members of the TSSCPP family in Lemma~\ref{lem:permconfig} and Theorem~\ref{thm:permmagog}.

\subsection{Identification of permutation TSSCPP}
\label{sec:conseq}

\begin{definition}
\label{def:permbool}
Let \emph{permutation TSSCPP} of order $n$ be all TSSCPP inside a $2n\times 2n\times 2n$ box whose corresponding  boolean triangles have weakly decreasing rows. 
\end{definition}

The terminology \emph{permutation TSSCPP} is justified by Proposition~\ref{prop:nfact} below, which counts permutation TSSCPP. In Theorem~\ref{thm:tsscppsn}, we further justify this terminology by providing a statistic-preserving bijection between permutations and permutation TSSCPP. 
\begin{proposition}
\label{prop:nfact}
There are $n!$ permutation TSSCPP of order $n$.
\end{proposition}

\begin{proof}
 The condition on a permutation TSSCPP boolean triangle that the rows be weakly decreasing means that all the ones must be left-justified. Thus the defining partial sum inequality of Definition~\ref{def:bool}
 is never violated. To construct a permutation TSSCPP, freely choose any number of left-justified ones in each row of the boolean triangle and the rest zeros; there are $i+1$ choices for row $i$, and the number of ones chosen in each row is independent of the choice for any other row.
\end{proof}

Though the viewpoint of boolean triangles is the most natural for characterizing permutation TSSCPP, we can use the TSSCPP family bijections discussed in the previous section to determine the permutation condition both on the magog triangle and directly on the plane partition. 

\begin{lemma}
\label{lem:permconfig}
Permutation TSSCPP, as integer arrays $t_{ij}$, $1\leq i,j\leq 2n$, are characterized as the TSSCPP such that there is no choice of integer $k\geq 0$ and fundamental domain indices $(i,j)$, $n+1\leq i\leq j\leq 2n-1$, 
such that \[t_{i,j} > t_{i,j+1} = t_{i+k,j+k+1} > t_{i+k+1,j+k+1}.\] 

Alternatively, as stacks of cubes in a corner, permutation TSSCPP are the TSSCPP with no configurations of the following form inside the fundamental domain:
\begin{center}
\includegraphics[scale=.75]{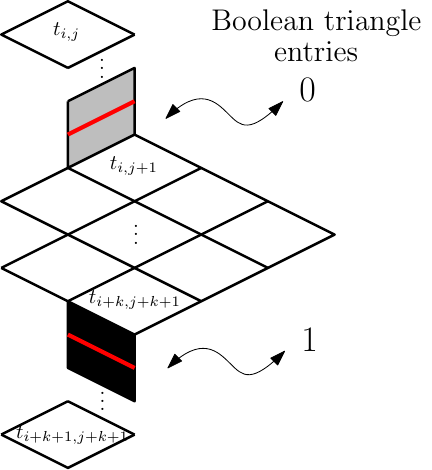} 
\end{center}

\end{lemma}

\begin{proof}
In the bijection of Proposition~\ref{prop:boolbij} (via Proposition~\ref{prop:nilp}), the left-to-right rows of a TSSCPP boolean triangle correspond to the top-to-bottom columns of steps in the fundamental domain. The non-permutation TSSCPP boolean triangle configuration in which a zero is to the left of a one in a row corresponds directly with the configuration shown above.
\end{proof}

Lemma~\ref{lem:permconfig} shows that the non-permutation TSSCPP in Figures~\ref{fig:cl10} and \ref{ex:tikztsscpp} is the one in the upper right.

\medskip
We can also adapt this condition to characterize permutation magog triangles directly. Note that this characterization of permutations is not at all obvious from the perspective of magog triangles, but relies heavily on the connection to boolean triangles. 
\begin{theorem}
\label{thm:permmagog}
Permutation magog triangles are characterized by the following condition:
there is no entry $\alpha_{i,j}$ and $k\geq 0$ such that
\[\alpha_{i,j} \geq \alpha_{i+1,j} = \alpha_{i+k+1,j-k} > \alpha_{i+k+1,j-k-1}+1.\]
\end{theorem}

\begin{proof}
From Proposition~\ref{prop:magog} we have that $\alpha_{i,n-j}=t_{n+j,n+i}+i-j+1$, so solving for $t_{n+j,n+i}$ and substituting $n+j\rightarrow i$ and $n+i\rightarrow j$, we have $t_{i,j}=\alpha_{j-n,2n-i}+i-j-1$. Thus the condition of Lemma~\ref{lem:permconfig} that there is no choice of integer $k\geq 0$ and fundamental domain indices $(i,j)$, $n+1\leq i\leq j\leq 2n-1$, 
such that \[t_{i,j} > t_{i,j+1} = t_{i+k,j+k+1} > t_{i+k+1,j+k+1}\] translates to 
\[\alpha_{j-n,2n-i}+i-j-1 > \alpha_{j+1-n,2n-i}+i-j-2\] 
\[= \alpha_{j+k+1-n,2n-i-k}+i-j-2 > \alpha_{j+k+1-n,2n-i-k-1}+i-j-1.\] 
Simplified, this is \[\alpha_{j-n,2n-i} > \alpha_{j+1-n,2n-i}-1\] 
\[= \alpha_{j+k+1-n,2n-i-k}-1 > \alpha_{j+k+1-n,2n-i-k-1}.\] 
Translating $j-n\rightarrow i$ and $2n-i\rightarrow j$ so that $j\rightarrow i+n$ and $i\rightarrow 2n-j$, we have
\[\alpha_{i,j} > \alpha_{i+1,j}-1 = \alpha_{i+k+1,j-k}-1 > \alpha_{i+k+1,j-k-1},\] which is equivalent to the desired condition.
\end{proof}

We will use this characterization of permutation magog triangles in Section~\ref{sec:poset} to define a new partial order on permutations via the natural partial order on magog triangles.

\subsection{The bijection}
\label{subsec:bij}
We are now ready to state and prove our main theorem.

\begin{theorem}
\label{thm:tsscppsn}
There is a natural, statistic-preserving bijection between $n\times n$ permutation matrices with inversion number $p$  whose one in the last row is in column $k$ and whose one in the last column is in row $\ell$ and permutation TSSCPP boolean triangles of order~$n$ with $p$ zeros, exactly $n-k$ of which are contained in the last row, and for which the lowest one in diagonal $n-1$ is in row $\ell-1$, or if $\ell=1$, then diagonal $n-1$ does not contain a one. 
\end{theorem}

\begin{proof}
We first describe the bijection map. An example of this bijection is shown in Figure~\ref{fig:bijex}.

Begin with a permutation TSSCPP of order $n$. Consider its associated boolean triangle $b=\{b_{i,j}\}$ for $1\leq i\leq n-1$, $n-i\leq j\leq n-1$. Define $a=\{a_{i,j}\}$ for $1\leq i\leq n$, $n-i\leq j\leq n-1$ as follows:
$a_{n,j}=j+1$ and for $i<n$, $a_{i,j}=a_{i+1,j}$ if $b_{i,j}=0$ and $a_{i,j}=a_{i+1,j-1}$ if $b_{i,j}=1$.  
We claim $a$ is a monotone triangle. Clearly $a_{i,j-1}\leq a_{i-1,j}\leq a_{i,j}$. Also, $a_{i,j}< a_{i,j+1}$, since if $a_{i,j} = a_{i,j+1}$, then $a_{i,j} = a_{i+1,j}$  and $a_{i,j+1}=a_{i+1,j}$ so that we would need $b_{i,j}=0$ and $b_{i,j+1}=1$. This contradicts the fact that the rows of permutation boolean triangles must weakly decrease. Furthermore, $a$ is a monotone triangle with no negative ones in the corresponding ASM, since each entry is defined to be equal to one of its diagonal neighbors in the row below; see Corollary~\ref{cor:negativeone}. This process is clearly invertible. 

\begin{figure}[htbp]
\begin{center}
\scalebox{.9}{$\begin{array}{c}
\mbox{\large{TSSCPP}}\\
\includegraphics[scale=.5]{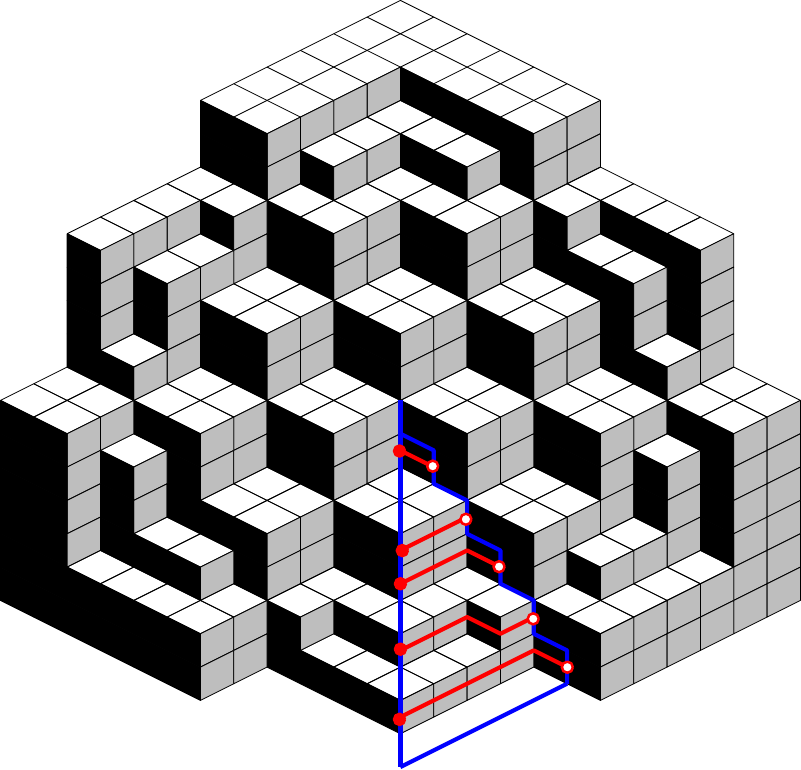}
\end{array}
\Leftrightarrow$}

\scalebox{.75}{
$\begin{array}{cc}
 \begin{array}{c}
\mbox{\LARGE{Non-intersecting}}\\
\mbox{\LARGE{lattice paths}}\\
\raisebox{\depth}{\scalebox{1}[-1]{
\scalebox{0.8}{
$\xymatrix @-1.0pc {
\circ\bullet & \bullet \ar@{-}[d] & \cdot & \cdot & \bullet \ar@{-}[dl]  & \bullet \ar@{-}[dl]  & \cdot & \bullet \ar@{-}[dl]  & \cdot   & \bullet \ar@{-}[dl] & \cdot \\
 & \circ & \cdot & \cdot \ar@{-}[dl] & \cdot \ar@{-}[dl]  & \cdot & \cdot \ar@{-}[dl]  & \cdot & \cdot \ar@{-}[dl]  & \cdot \\
 &  & \circ & \cdot \ar@{-}[d]  & \cdot & \cdot \ar@{-}[d]  & \cdot & \cdot \ar@{-}[dl]  & \cdot &  \\
 &  &  & \circ & \cdot & \cdot \ar@{-}[dl]  & \cdot \ar@{-}[dl]  & \cdot &  & \Large{\Leftrightarrow} \\
& & &  & \circ & \cdot \ar@{-}[d] & \cdot & &  &  \\
& & & &  & \circ 
}$}
}}
\end{array}
&
\begin{array}{c}
\mbox{\LARGE{Boolean triangle}}\\
\mbox{ }\\
\scalebox{.9}{$
\begin{array}{ccccccccc}
   & & & & 1 & & & &  \\
   & & & \Large{\textcolor{red}{\textbf{0}}} & & \Large{\textcolor{red}{\textbf{0}}} & & &  \Leftrightarrow \\
   & & 1 & & 1 & & \Large{\textcolor{red}{\textbf{0}}} & &  \\
   & \Large{\textcolor{red}{\textbf{0}}} & & \Large{\textcolor{red}{\textbf{0}}} & & \Large{\textcolor{red}{\textbf{0}}} & & \Large{\textcolor{red}{\textbf{0}}} &  \\
   1 & & \Large{\textcolor{red}{\textbf{0}}} & & \Large{\textcolor{red}{\textbf{0}}} & & \Large{\textcolor{red}{\textbf{0}}} & & \Large{\textcolor{red}{\textbf{0}}} \end{array}$} \end{array}
\\
  &\\
\begin{array}{c}
\mbox{\LARGE{Monotone triangle}}\\
\mbox{ }\\
\begin{array}{ccccccccccc}
  & & & & & 4 & & & & & \\
  & & & & \Large{\textcolor{red}{\textbf{4}}} & & \Large{\textcolor{red}{\textbf{6}}} & & & & \Leftrightarrow\\
  & & & 3 & & 4 & & \Large{\textcolor{red}{\textbf{6}}} & & & \\
  & & \Large{\textcolor{red}{\textbf{3}}} & & \Large{\textcolor{red}{\textbf{4}}} & & \Large{\textcolor{red}{\textbf{5}}} & & \Large{\textcolor{red}{\textbf{6}}} & & \\
  & 1 & & \Large{\textcolor{red}{\textbf{3}}} & & \Large{\textcolor{red}{\textbf{4}}} & & \Large{\textcolor{red}{\textbf{5}}} & & \Large{\textcolor{red}{\textbf{6}}} &\\
  1 & & 2 & & 3 & & 4 & & 5 & & 6\end{array} \end{array}
&
\begin{array}{c}
\mbox{\LARGE{Permutation matrix}}\\
\mbox{ }\\
\left( \begin{array}{cccccc}
0&0&0&1&0&0\\
0&0&0&0&0&1\\
0&0&1&0&0&0\\
0&0&0&0&1&0\\
1&0&0&0&0&0\\
0&1&0&0&0&0
\end{array}\right) \end{array}
\end{array}$}
\end{center}
\caption{An example of the bijection.  Note that the matrix on the right represents the permutation $463512$ which has $11$ inversions. These inversions correspond to the $11$ entries in the monotone triangle equal to their southeast diagonal neighbor (shown in bold), as well as the zeros in the boolean triangle and the diagonal steps of the nest of non-intersecting lattice paths.}
\label{fig:bijex}
\end{figure}

We now show that this map takes a permutation TSSCPP boolean triangle with $p$ zeros to a permutation matrix with $p$ inversions.
Note that to convert from the monotone triangle representation of a permutation to usual one-line notation $\sigma$ such that $i\rightarrow \sigma(i)$, we set $\sigma(i)$ equal to the unique new value in row $i$ of the monotone triangle. Thus for each entry of the monotone triangle $a_{i,j}$ such that $a_{i,j}=a_{i+1,j}$, there will be an inversion in the permutation between $a_{i,j}$ and $\sigma(i+1)$. This is because $a_{i,j}=\sigma(k)$ for some $k\leq i$ and $\sigma(k)=a_{i,j}>\sigma(i+1)$. These entries $a_{i,j}$ such that $a_{i,j}=a_{i+1,j}$ correspond exactly to zeros in row $i$ of the boolean triangle $b$. Thus if a permutation TSSCPP has $p$ zeros in its boolean triangle, its corresponding permutation will have $p$ inversions.

Also, observe that if the number of zeros in the last row of the boolean triangle is $k$, then the one in the bottom row of the permutation matrix will be in column $n-k$. So the missing number in row $n-1$ of the monotone triangle determines where the last row of the boolean triangle transitions from ones to zeros. So by the bijection between monotone triangles and ASM, the one in the last row of the resulting permutation matrix is in column $n-k$.

Finally, if the lowest one in diagonal $n-1$ of the boolean triangle is in row $\ell-1$, this means that the entries $\{a_{i,n-1}\}$ for $\ell\leq i\leq n$ are all equal to $n$. So the one in the last column of the permutation matrix is in row $\ell$.
\end{proof}

See Figure~\ref{fig:bijex} for an example of this bijection. 

%%%%%%%%%%%%%%%%%%%%%%%%%%%%%%%%%%%%%%%%%%%%%%%%%%%%%%%%%%%%
\subsection{Outlook on a general bijection}
%%%%%%%%%%%%%%%%%%%%%%%%%%%%%%%%%%%%%%%%%%%%%%%%%%%%%%%%%%%%
\label{sec:spec}

In~\cite{Striker_DPP}, we discussed the obstacles to turning the bijection between permutations and descending plane partitions presented there into a bijection between all ASM and DPP. Here we discuss some of the challenges to the ASM-TSSCPP bijection in full generality. 

While DPP have the property that the number of parts equals the inversion number of the ASM (this is now proved, though not bijectively~\cite{ZINNDPP}), TSSCPP do not have such a statistic as of yet. We showed that the number of 
zeros in a permutation TSSCPP boolean triangle gives the inversion number of the permutation matrix, but this is not true for general TSSCPP boolean triangles and ASM. This is because the distribution of the number of zeros over all TSSCPP boolean triangles of order $n$ does not correspond to the inversion number distribution on $n\times n$ ASM (using either the inversion number of Definition~\ref{def:inv} or the alternative definition discussed in the paragraph afterward).  Furthermore, while the number of special parts of a DPP corresponds to the number of negative ones in the ASM~\cite{ZINNDPP}, no such TSSCPP statistic is currently known. It would seem reasonable to conjecture that the negative ones of the ASM should correspond to all instances of 
a zero followed by a one as you go across a row of the boolean triangle. We have calculated in Sage~\cite{sage} that this is true up to $n=4$. For $n\geq 5$ this correspondence seems to hold in the special cases of one negative one and the maximum number of negative ones $\lfloor\frac{(n-1)^2}{4}\rfloor$, but for the number of negative ones between one and  $\lfloor\frac{(n-1)^2}{4}\rfloor$, these statistics diverge.

T.~Fonseca and P.~Zinn-Justin have shown (as a singly-refined specialization of the theorem in Section 3.3 of~\cite{FonsecaZinnJustin})
that the distribution of the number of diagonal steps in the bottom row of the TSSCPP non-intersecting lattice paths corresponds to the enumeration of ASM refined by the position of the unique one in the top row. This was first conjectured in terms of magog triangles by W.~Mills, D.~Robbins, and H.~Rumsey (as the $k=1$ case of Conjecture 2 in~\cite{MRR3}) and mentioned in terms of lattice paths by P.~Zinn-Justin and P.~Di Francesco in \cite{difran_zinn08}. In view of this correspondence, one might hope to begin a general bijection by determining the $(n-1)$st row of the monotone triangle from the bottom row of the TSSCPP boolean triangle by left-justifying all the vertical steps and then bijecting in the same way as in the permutation case. After that, though, it is unclear how to proceed. See Figure~\ref{fig:stat2} for a summary of the various statistics which are preserved in the permutation case DPP-ASM-TSSCPP bijections and which should be true in full generality since the distributions correspond. (See~\cite{Striker_DPP} for further explanation of the DPP case.)

\begin{figure}[htbp]
\begin{center}
\begin{tabular}{|c|c|c|}
\hline
DPP&ASM&TSSCPP\\
&& boolean triangle\\
\hline
no special parts* & no $-1$'s & rows weakly decrease\\
\hline
number of parts* & number of inversions & number of zeros\\
\hline
number of $n$'s* & position of 1  & position of lowest 1\\
& in last column & in last diagonal\\
\hline
largest part value  & position of 1  & number of zeros \\
 not appearing & in last row & in last row*\\
\hline
\end{tabular}
\end{center}
\caption{This table shows the statistics preserved by the permutation case bijections of this paper and~\cite{Striker_DPP}. There is a star by the DPP and TSSCPP statistics that have the same distribution as the ASM statistic in the general case. This is proved for DPP in~\cite{ZINNDPP} and  for TSSCPP in~\cite{FonsecaZinnJustin}.}
\label{fig:stat2}
\end{figure}

We now compare this work with other bijections between subsets of ASM and TSSCPP.
% due to P.~Biane and H.~Cheballah. 
In~\cite{Biane_Cheballah_1}, the authors give a bijection between \emph{monotone trapezoids} and \emph{magog trapezoids} of two diagonals. 
The term \emph{trapezoid} indicates the truncation of the monotone or magog triangle to a fixed number of diagonals. Their bijection is both more and less general than the one of this paper. It is more general in the sense that it includes configurations corresponding to the negative one in an ASM, whereas we consider only permutations. It is less general in that it uses only two diagonals of the triangle, where we are able to consider the full triangle corresponding to a permutation TSSCPP. More recently, J.~Bettinelli has found a simpler bijection in this case~\cite{Bettinelli}.
%Experimental evidence suggests the bijection of~\cite{Biane_Cheballah_1} and the bijection of this paper may coincide (up to slight deformation) in the case of permutation monotone triangles, truncated to two diagonals. 
Perhaps the combination of these various perspectives will provide insight on the full bijection.

\section{Poset structures and Catalan subposets}
\label{sec:poset}
In this last section, we compare various partial orders on ASM and TSSCPP, along with their permutation and Catalan subposets. In Section~\ref{subsec:AnTn}, we review distributive lattices $A_n$ of ASM and $T_n$ of TSSCPP, defined naturally on their corresponding monotone or magog triangles~\cite{ELKP_DOMINO1}~\cite{STRIKERPOSET}. We note that the permutation subposet of $A_n$ is known to be the strong Bruhat order~\cite{TREILLIS}. In Section~\ref{subsec:tamari}, we use the characterization of permutation TSSCPP in Theorem~\ref{thm:tsscppsn} to define and study the permutation subposet $T_n^{\rm Perm}$ of $T_n$. While $T_n^{\rm Perm}$ does not coincide with any of the well-studied permutation posets, we show in Theorems~\ref{thm:tamari} and \ref{thm:Catdistr} that $T_n^{\rm Perm}$ contains two different Catalan subposets: the Tamari lattice and the Catalan distributive lattice. In Section~\ref{subsec:TBool}, we define a new TSSCPP partial order $TBool_n$ by componentwise comparison of the TSSCPP boolean triangles. We show in Corollary~\ref{cor:TBool} that its permutation subposet $TBool_n^{\rm Perm}$ is an especially nice product of chains distributive lattice which sits between the weak and strong Bruhat orders. Finally, we show in Corollary~\ref{cor:tamcat2} that $TBool_n^{\rm Perm}$ contains the same Tamari and Catalan  subposets as $T_n^{\rm Perm}$.

\subsection{ASM and TSSCPP posets via monotone and magog triangles}
\label{subsec:AnTn}
We start by reviewing natural ASM and TSSCPP posets defined via their monotone and magog triangles. The ASM poset was first introduced in \cite{ELKP_DOMINO1}  and further studied in~\cite{TREILLIS} and \cite{STRIKERPOSET}. This partial order can be naturally and equivalently defined using any of the following objects in bijection with ASM: monotone triangles, corner sum matrices, or height functions. Here we give the definition using monotone triangles. See Figure~\ref{fig:A3T3}, left, for an example.

\begin{definition}
\label{def:An}
Let  $A_n$, the ASM poset of order $n$, be the partial order defined by componentwise comparison of the monotone triangles of order $n$.
\end{definition}

\begin{figure}[htbp]
\begin{center}
\scalebox{.75}{
\begin{tikzpicture}[>=latex,line join=bevel,]
\node (node_6) at (77.500000bp,21.500000bp) [draw,draw=none] {$\begin{array}{ccccc}
&& 1 &\\
&1 & & 2\\
1&&2&&3
\end{array}$};
  \node (node_5) at (121.500000bp,100.500000bp) [draw,draw=none] {$\begin{array}{ccccc}
&& 1 &\\
&1 & &3\\
1&&2&&3
\end{array}$};
  \node (node_4) at (34.500000bp,100.500000bp) [draw,draw=none] {$\begin{array}{ccccc}
&& 2 &\\
&1 & &2\\
1&&2&&3
\end{array}$};
  \node (node_3) at (77.500000bp,179.500000bp) [draw,draw=none] {$\begin{array}{ccccc}
&& \Large{\textcolor{blue}{\textbf{2}}} &\\
&\Large{\textcolor{blue}{\textbf{1}}} & & \Large{\textcolor{blue}{\textbf{3}}}\\
\Large{\textcolor{blue}{\textbf{1}}}&&\Large{\textcolor{blue}{\textbf{2}}}&&\Large{\textcolor{blue}{\textbf{3}}}
\end{array}$};
  \node (node_2) at (34.500000bp,258.500000bp) [draw,draw=none] {$\begin{array}{ccccc}
&& 3 &\\
&1 & &3\\
1&&2&&3
\end{array}$};
  \node (node_1) at (121.500000bp,258.500000bp) [draw,draw=none] {$\begin{array}{ccccc}
&& 2 &\\
&2 & &3\\
1&&2&&3
\end{array}$};
  \node (node_0) at (77.500000bp,337.500000bp) [draw,draw=none] {$\begin{array}{ccccc}
&& 3 &\\
&2 & &3\\
1&&2&&3
\end{array}$};
  \draw [black,->] (node_3) ..controls (94.026000bp,209.420000bp) and (99.571000bp,219.120000bp)  .. (node_1);
  \draw [black,->] (node_3) ..controls (61.398000bp,209.330000bp) and (56.043000bp,218.920000bp)  .. (node_2);
  \draw [black,->] (node_2) ..controls (50.602000bp,288.330000bp) and (55.957000bp,297.920000bp)  .. (node_0);
  \draw [black,->] (node_5) ..controls (104.970000bp,130.420000bp) and (99.429000bp,140.120000bp)  .. (node_3);
  \draw [black,->] (node_6) ..controls (94.026000bp,51.420000bp) and (99.571000bp,61.125000bp)  .. (node_5);
  \draw [black,->] (node_6) ..controls (61.398000bp,51.334000bp) and (56.043000bp,60.923000bp)  .. (node_4);
  \draw [black,->] (node_1) ..controls (104.970000bp,288.420000bp) and (99.429000bp,298.120000bp)  .. (node_0);
  \draw [black,->] (node_4) ..controls (50.602000bp,130.330000bp) and (55.957000bp,139.920000bp)  .. (node_3);
\end{tikzpicture}
\begin{tikzpicture}[>=latex,line join=bevel,]
\node (node_6) at (77.500000bp,21.500000bp) [draw,draw=none] {$\begin{array}{ccccc}
&& 1 &\\
&1 & & 2\\
1&&2&&3
\end{array}$};
  \node (node_4) at (34.500000bp,100.500000bp) [draw,draw=none] {$\begin{array}{ccccc}
&& 2 &\\
&1 & &2\\
1&&2&&3
\end{array}$};
  \node (node_3) at (77.500000bp,179.500000bp) [draw,draw=none] {$\begin{array}{ccccc}
&& 2 &\\
&1 & & 3\\
1&&2&&3
\end{array}$};
  \node (node_8) at (-9.000000bp,179.500000bp) [draw,draw=none] {$\begin{array}{ccccc}
&& 3 &\\
&1 & & 2\\
1&&2&&3
\end{array}$};
  \node (node_2) at (34.500000bp,258.500000bp) [draw,draw=none] {$\begin{array}{ccccc}
&& \Large{\textcolor{blue}{\textbf{3}}} &\\
&\Large{\textcolor{blue}{\textbf{1}}} & &\Large{\textcolor{blue}{\textbf{3}}}\\
\Large{\textcolor{blue}{\textbf{1}}}&&\Large{\textcolor{blue}{\textbf{2}}}&&\Large{\textcolor{blue}{\textbf{3}}}
\end{array}$};
  \node (node_1) at (121.500000bp,258.500000bp) [draw,draw=none] {$\begin{array}{ccccc}
&& 2 &\\
&2 & &3\\
1&&2&&3
\end{array}$};
  \node (node_0) at (77.500000bp,337.500000bp) [draw,draw=none] {$\begin{array}{ccccc}
&& 3 &\\
&2 & &3\\
1&&2&&3
\end{array}$};
  \draw [black,->] (node_3) ..controls (94.026000bp,209.420000bp) and (99.571000bp,219.120000bp)  .. (node_1);
  \draw [black,->] (node_3) ..controls (61.398000bp,209.330000bp) and (56.043000bp,218.920000bp)  .. (node_2);
  \draw [black,->] (node_2) ..controls (50.602000bp,288.330000bp) and (55.957000bp,297.920000bp)  .. (node_0);
  \draw [black,->] (node_6) ..controls (61.398000bp,51.334000bp) and (56.043000bp,60.923000bp)  .. (node_4);
  \draw [black,->] (node_1) ..controls (104.970000bp,288.420000bp) and (99.429000bp,298.120000bp)  .. (node_0);
  \draw [black,->] (node_4) ..controls (50.602000bp,130.330000bp) and (55.957000bp,139.920000bp)  .. (node_3);
    \draw [black,->] (node_4) ..controls (19.00000bp,130.330000bp) and (12.043000bp,139.920000bp)  .. (node_8);
      \draw [black,->] (node_8) ..controls (5.043000bp,209.420000bp) and (12.571000bp,219.120000bp)  .. (node_2);
\end{tikzpicture}
}
\end{center}
\caption{Left: The poset $A_3$ of  monotone triangles of order $3$, partially ordered by componentwise comparison; Right: The poset $T_3$ of magog triangles of order $3$, partially ordered by componentwise comparison. The highlighted element in each poset is the non-permutation.}
\label{fig:A3T3}
\end{figure}
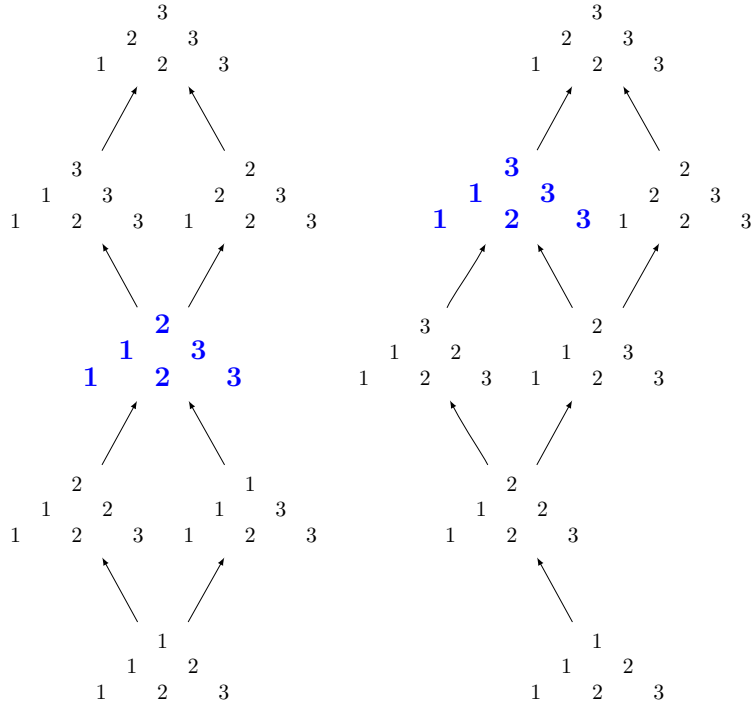

$A_n$ is a \emph{lattice}, which means any pair of elements has a unique \emph{meet} (greatest lower bound) and \emph{join} (least upper bound); $A_n$ is, furthermore, a \emph{distributive lattice}, meaning the operations of meet and join \emph{distribute} over each other. The \emph{fundamental theorem of finite distributive lattices} says that $A_n$ is, thus, isomorphic to the \emph{lattice of order ideals} $J(P_n)$ for some \emph{poset of join irreducibles} $P_n$. (An \emph{order ideal} is a subset $X$ of poset elements such that if $p\in X$ and $q\leq p$ then $q\in X$. See Chapter 3 of~\cite{Stanley_1} for poset terminology.) This join-irreducible poset $P_n$ has a nice structure, which we describe in the following theorem in a different, but equivalent, way to~\cite{ELKP_DOMINO1} and \cite{STRIKERPOSET}; see Figure~\ref{fig:A3T3}, left. See Definition 8.4 of~\cite{prorow} for an equivalent construction of $P_n$ as a layering of successively smaller \emph{Type A positive root posets}; see also the constructions in~\cite{ELKP_DOMINO1}, \cite{ProppManyFaces},  \cite{STRIKERPOSET}, and \cite{razstrogrow}.

\begin{theorem}[\cite{ELKP_DOMINO1}, \cite{STRIKERPOSET}]
\label{thm:An}
$A_n$ is isomorphic to the distributive lattice of order ideals of the poset $P_n$ described below.
Let the elements of $P_n$ be the coordinates $(i,j,k)$ in $\mathbb{Z}^3$ such that $0\leq i\leq n-2$, $0\leq j\leq n-2-i$, and $0\leq k\leq n-2-i-j$. Let the covering relations of $P_n$ be given as: $(i,j,k)$ covers $(i,j+1,k)$, $(i,j+1,k-1)$, $(i+1,j,k)$, and $(i+1,j,k-1)$, whenever these coordinates are poset elements. 
\end{theorem}

\begin{figure}[htbp]
\begin{center}
\includegraphics[scale=.45]{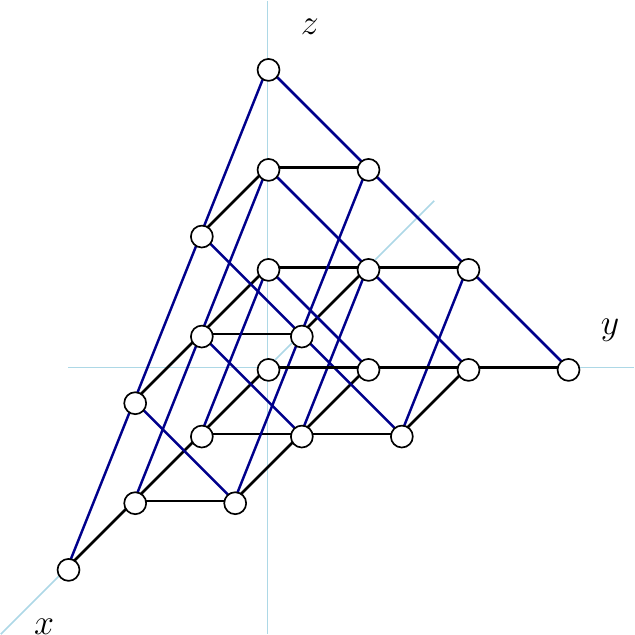}
\includegraphics[scale=.45]{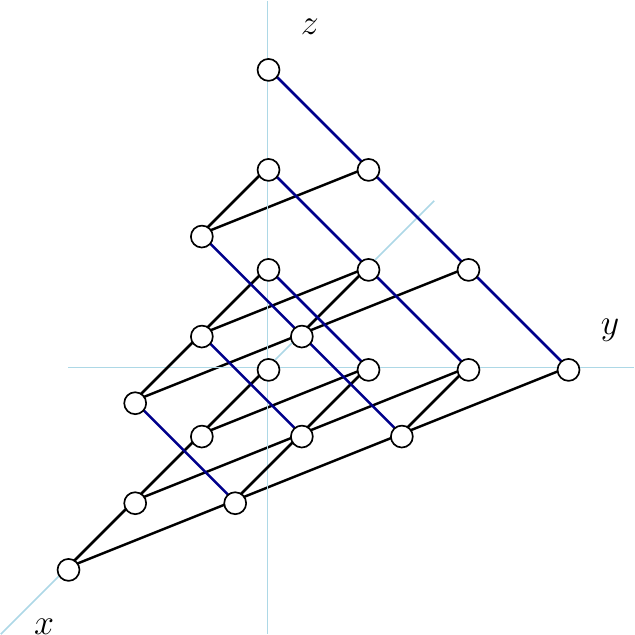}
\end{center}
\caption{Left: The poset $P_5$ such that $A_5\cong J(P_5)$; Right: The poset $Q_5$ such that $T_5\cong J(Q_5)$}
\label{fig:P5A5}
\end{figure}

The permutation subposet of $A_n$ is known to be a well-studied partial order on the symmetric group; we state here a theorem of A.~Lascoux and M.-P.\ Sch{\"u}tzenberger~\cite{TREILLIS} about the permutation subposet of $A_n$ (see Definition~\ref{def:bruhat} for the definition of Bruhat order).

\begin{theorem}[\cite{TREILLIS}]
\label{thm:MacNeille}
The induced subposet of $A_n$ on permutations of $n$, which we denote as $A_n^{\rm Perm}$, is isomorphic to the strong Bruhat order. Moreover, $A_n$ is the smallest lattice to contain the strong Bruhat order as a subposet, that is, $A_n$ is the MacNeille completion of the strong Bruhat order.
\end{theorem}

We now recall a natural partial order on TSSCPP. See Figure~\ref{fig:A3T3}, right, for an example.
\begin{definition}
\label{def:Tn}
Let $T_n$, the TSSCPP poset of order $n$, be the partial order defined by componentwise comparison of the magog triangles of order $n$.
\end{definition}

Like $A_n$, $T_n$ is also a distributive lattice. In \cite{STRIKERPOSET}, we showed that $A_n$ and $T_n$ have similar posets of join irreducibles. We also described these posets using colored vectors in $\mathbb{R}^3$ for the various different types of covering relations; we furthermore defined and studied the \emph{tetrahedral poset family} of which these ASM and TSSCPP posets are members. Instead of delving into that theory here, in Theorem~\ref{thm:An} above and Theorem~\ref{thm:Tn} below we describe these posets in a different, but equivalent, way; see Figure~\ref{fig:P5A5}.
\begin{theorem}[\cite{STRIKERPOSET}]
\label{thm:Tn}
$T_n$ is  isomorphic to the distributive lattice of order ideals of the poset $Q_n$ described below.
Let the elements of $Q_n$ be the coordinates $(i,j,k)$ in $\mathbb{Z}^3$ such that $0\leq i\leq n-2$, $0\leq j\leq n-2-i$, and $0\leq k\leq n-2-i-j$. Let the covering relations of $Q_n$ be given as: $(i,j,k)$ covers $(i+1,j-1,k)$, $(i,j+1,k-1)$, and $(i+1,j,k)$, whenever these coordinates are poset elements. 
\end{theorem}

In the next subsection, we study a new partial order on permutations, namely, the induced subposet of $T_n$ on permutation TSSCPP, as a consequence of the identification of permutation TSSCPP in Theorem~\ref{thm:tsscppsn}. 

\subsection{The magog permutation poset and its Catalan subposets}
\label{subsec:tamari}
Using the characterization of permutation magog triangles in Theorem~\ref{thm:permmagog}, we introduce a new partial order on permutations as the restriction of $T_n$ to permutation magog triangles. See Figure~\ref{fig:A3T3}, right, with the non-permutation removed and Figure~\ref{fig:Tpermmagog34}.

\begin{definition}
\label{def:permmagog}
Let the \emph{magog permutation poset} $T_n^{\rm Perm}$ on permutations of $n$ be given by componentwise comparison of the corresponding magog triangles. That is, $T_n^{\rm Perm}$ is the induced subposet of $T_n$ consisting of the \emph{permutation} magog triangles.
\end{definition}

Though $T_n^{\rm Perm}$ does not correspond to any of the usual permutation posets (see Figure~\ref{fig:Tpermmagog34}), we show in Theorems~\ref{thm:tamari} and \ref{thm:Catdistr} that $T_n^{\rm Perm}$ contains two different Catalan posets as induced subposets: the Tamari lattice on $132$-avoiding permutations and the Catalan distributive lattice on $213$-avoiding permutations. 
Before stating and proving these theorems on subposets of $T_n^{\rm Perm}$, we note in the following proposition a structural property of $T_n^{\rm Perm}$. 

\begin{proposition}
\label{prop:notlatticemagog}
$T_n^{\rm Perm}$ is a lattice for $n\leq 3$, but for $n\geq 4$ it is not a lattice.
\end{proposition}

\begin{proof}
We have computed $T_n^{\rm Perm}$ in Sage~\cite{sage} for $n\leq 5$. $T_3^{\rm Perm}$ is the (non-distributive) lattice in Figure~\ref{fig:Tpermmagog34}, left. $T_4^{\rm Perm}$, shown in Figure~\ref{fig:Tpermmagog34}, right, is not a lattice. For any $n>1$, $T_{n-1}^{\rm Perm}$ sits as an interval inside $T_n^{\rm Perm}$. Thus if $T_{n-1}^{\rm Perm}$ is not a lattice, then neither is $T_n^{\rm Perm}$.
\end{proof}

We will find it easier to prove Theorems~\ref{thm:tamari} and \ref{thm:Catdistr} if we work with $T_n^{\rm Perm}$ from the boolean triangle perspective. So we first characterize the covering relations in $T_n$ in terms of boolean triangles rather than magog triangles. 
\begin{lemma}
\label{lem:magogboolcovers}
Suppose $t'$ covers $t$ in $T_n$. Let $b$ and $b'$ be the TSSCPP boolean triangles corresponding to $t$ and $t'$, respectively. Then $b$ and $b'$ differ only in one of the following:
\begin{enumerate}
\item $b'_{i,j} = 0 = b_{i+1,j}$ and $b'_{i+1,j} = 1 = b_{i,j}$ for some $1\leq i\leq n-2$, $n-i \leq j\leq  n-1$, or
\item $b'_{n-1,j} = 0$ and $b_{n-1,j} = 1$ for some $1\leq j\leq n-1$.
\end{enumerate}
That is, to create $b'$ from $b$, either {\rm (1)} swap a one with its southeast diagonal neighboring zero, or {\rm (2)} replace a one in the bottom row with a zero.
\end{lemma}

\begin{proof}
The magog covering relations on TSSCPP (as stacks of cubes in a corner) consist in a box being added within the fundamental domain. This means in the nest of non-intersecting lattice paths, either (1) two adjacent steps of different types swap to add the box, or (2) a box is added along the center border of the fundamental domain, in which case a step at the end of a path changes from one type to the other. These correspond to cases 1.\ and 2.\ above.
\end{proof}

The Tamari lattice is a natural partial order often defined on the Catalan objects binary trees or triangulations of a convex polygon. The covering relations on binary trees are given by a local move called \emph{tree rotation}; the triangulation covering relations are given by \emph{flipping a diagonal}. See~\cite{tamariref} for further references. Rather than using the definition in terms of binary trees or triangulations, we will use the original definition, often called the \emph{bracket vector encoding}, given by S.~Huang and D.~Tamari in~\cite{HuangTamari}. See Figure~\ref{fig:Cat_compare}, left.

\begin{definition}[\cite{HuangTamari}]
\label{def:tamari}
Let the \emph{Tamari lattice of order $n$}, which we denote as $Tam_n$, be defined as the set of sequences of $n$ positive integers $x_1, x_2,\ldots, x_n$, such that $i\leq x_i \leq n$ and if $i\leq j\leq x_i$ then $x_j\leq x_i$, ordered by reverse componentwise comparison. Such sequences are called \emph{bracket vectors}.
\end{definition}

Note that for convenience, we use the reverse componentwise comparison, but as the Tamari lattice is self-dual, this convention is inconsequential. 

\medskip
We show in Theorem~\ref{thm:tamari} below that $T_n^{\rm Perm}$ contains $Tam_n$ as the induced subposet of permutations that \emph{avoid the pattern} $132$. So we will first need to define pattern avoidance in permutations.

\begin{definition}
\label{def:patterns}
The permutation $\pi=\pi(1) \pi(2) \ldots \pi(k)$ is \emph{contained} in the permutation $\sigma=\sigma(1)\sigma(2)\ldots\sigma(n)$, $k\leq n$, if there is a substring $\sigma(i_1), \sigma(i_2),\ldots, \sigma(i_k)$, $i_1<i_2<\cdots<i_k$ whose values have the same relative order as $\pi(1), \pi(2), \ldots, \pi(k)$. A permutation $\sigma$ \emph{avoids} the pattern $\pi$ if it does not contain~$\pi$.
\end{definition}

\begin{theorem}
\label{thm:tamari}
$T_n^{\rm Perm}$ contains a subposet which is isomorphic to $Tam_n$. In particular, the subposet of $T_n^{\rm Perm}$ consisting of the $132$-avoiding permutations is isomorphic to $Tam_n$.
\end{theorem}

\begin{proof}
Let $b$ be a permutation TSSCPP boolean triangle. Let $\Sigma_k:=\displaystyle\sum_{j=n-k}^{n-1} b_{k,j}$, that is $\Sigma_k$ denotes the sum of the $k$th row of $b$.
Define a sequence of positive integers $x_1,x_2,\ldots,x_n$ where $x_i = i + \Sigma_{n-i}$, and we consider $\Sigma_0=0$.
%the sum of the empty row zero to be zero. 
This is an injection from the set of permutation TSSCPP boolean triangles of order $n$ to the set of sequences of integers $x_1,x_2,\ldots,x_n$ with $i\leq x_i\leq n$ since the sum of row $n-i$ of $b$ can be at most $n-i$. 
%rows of a permutation TSSCPP boolean triangle must be decreasing. 

We wish to show that if the permutation associated with $b$ is $132$-avoiding, then $x_1,x_2,\ldots,x_n$ satisfies the inequalities in Definition~\ref{def:tamari} and is thus a bracket vector. As discussed above, $i\leq x_i\leq n$. 
Now we wish to show that the $132$-avoiding condition on $b$ corresponds to the property that if $i\leq j\leq x_i$ then $x_j\leq x_i$.
In the bijection of Theorem~\ref{thm:tsscppsn}, the permutation $132$ corresponds to the TSSCPP boolean triangle 
\scalebox{.8}{$\begin{array}{ccc}
& 1 &\\
1 & &0
\end{array}$}. 
More generally, we wish to show that $b$ corresponds to a permutation that is $132$-avoiding if and only if $b$ has the following property:

\emph{Property $(*)$:}  If $b_{i,j}=0$, at most one of the following may be true: $b_{i,j-1}=1$ or $b_{i',j}=1$ for some $i'< i$.

Suppose to the contrary that $b_{i,j}=0$, $b_{i,j-1}=1$, and let $i'$ be the largest number less than $i$ such that $b_{i',j}=1$ (suppose such an $i'$ exists). This means that the corresponding monotone triangle entries satisfy $a_{i+1,j}=a_{i,j}=\cdots = a_{i'+1,j}$, $a_{i,j-1}=a_{i+1,j-2}$, and $a_{i',j}=a_{i'+1,j-1}$. The equalities $a_{i,j}=a_{i+1,j}$ and $a_{i,j-1}=a_{i+1,j-2}$ together mean that $a_{i+1,j-1}$ must be the new value in row $i+1$, since the new value will always correspond to the entry horizontally between the ones and zeros in the previous line. So in the one-line notation $i\rightarrow \sigma(i)$ for the corresponding permutation, we have $\sigma(i+1)=a_{i+1,j-1}$. 
Now since $a_{i',j}=a_{i'+1,j-1}$, we know the new entry in row $i'+1$ is $a_{i'+1,j}$ or something larger, so $\sigma(i'+1)\geq a_{i'+1,j}=a_{i+1,j}>a_{i+1,j-1}=\sigma(i+1)$, where the last inequality is from the row-strict condition of the monotone triangle definition. Finally, $a_{i',j}=\sigma(i'')$ for some $i''\leq i'$, since $a_{i',j}$ must appear as the new value in row $i'$ or above. Now $a_{i',j}=a_{i'+1,j-1}$ since $b_{i',j}=1$, and $a_{i'+1,j-1}\leq a_{i+1,j-1}$ by the monotone triangle definition, so $\sigma(i'')=a_{i',j}\leq a_{i+1,j-1}=\sigma(i+1)$. Since $\sigma$ is a permutation and $i''\neq i+1$, the inequality must be strict. Therefore $\sigma(i'')<\sigma(i+1)<\sigma(i'+1)$ and $i''<i'+1<i+1$, so $\sigma(i'')$, $\sigma(i+1)$, $\sigma(i'+1)$ is our $132$ pattern.
Therefore, we have proven $(*)$ holds whenever $b$ corresponds to a $132$-avoiding permutation. %claim that if $b_{i,j}=0$ at most one of the following may be true: $b_{i,j-1}=1$ or $b_{i',j}=1$ for some $i'< i$. Call this $132$-avoiding condition on $b$ $(*)$.
Furthermore, any $132$ pattern in the permutation corresponding to $b$ will produce a violation of $(*)$ in the same way as discussed above. So $(*)$ holds if and only if the permutation corresponding to $b$ contains a $132$ pattern.

Recall we have set $x_i=i+\Sigma_{n-i}$.
We wish to show that the Tamari condition (if $i\leq j\leq x_i$ then $x_j\leq x_i$), which we denote as $(**)$, is equivalent to the $132$-avoiding condition $(*)$ on the corresponding boolean triangle $b$. Suppose $b$ satisfies $(*)$ and $i$ and $j$ are such that $i\leq j\leq x_i$. We wish to show $(**)$. 

Since $x_i\geq j$ we have  $i+\Sigma_{n-i}\geq j$ so that $\Sigma_{n-i}\geq j-i$.
Suppose $\Sigma_{n-i}=m\geq j-i$. Then $b_{n-i,\ell}$ equals $1$ for all $i\leq\ell\leq m+i-1$ and $0$ for all $\ell\geq m+i$, provided these entries exist. 
So $b_{n-i,m+i-1}=1$ and $b_{n-i,m+i}=0$, then by $(*)$, $b_{n-j,m+i}=0$ as well since $j>i$. Thus $\Sigma_{n-j} + j-i \leq \Sigma_{n-i}$, which immediately gives $x_j\leq x_i$. So $(*)$ implies $(**)$ on $x_1,\ldots,x_n$. 

Now suppose $x_1,\ldots,x_n$ satisfies $(**)$ and we define $b$ by $\Sigma_{n-i}=x_i-i$. This map is well defined since the entries of $b$ are $0$ or $1$ and the ones are all left-justified. Let $i\leq j\leq x_i$, so by $(**)$, $x_j\leq x_i$. Then $\Sigma_{n-j}+j\leq \Sigma_{n-i}+i$, so by inverting the above argument, we have $(*)$ on $b$.
Thus the $132$-avoiding permutation TSSCPP are in bijection with the bracket vectors of the Tamari lattice.

Finally, we must show that the reverse componentwise comparison covering relation of the bracket vectors $x_1,x_2,\ldots,x_n$ equals the magog partial order on $132$-avoiding permutation TSSCPP boolean triangles, as in Lemma~\ref{lem:magogboolcovers}. The covering relations on boolean triangles in any subposet of $T_n$ are given by composing the covering relations in Lemma~\ref{lem:magogboolcovers} until another element in the subposet is obtained.  Now by Lemma~\ref{lem:magogboolcovers}, $b'$ covers $b$ in $T_n$ if they differ in the following way: either a bottom-row zero of $b'$ turns into a one in $b$, or a one of $b'$ swaps with a zero to its northwest. Thus, $b'\geq b$ in $T_n^{\rm Perm}$ restricted to $132$-avoiding permutations if it is possible to transform $b'$ into $b$ by moving one or more ones northwest in their diagonal and/or adding ones in the bottom row. But note that a one cannot slide northwest in its diagonal leaving a zero behind, since that would violate the $132$-avoiding condition. So instead, if a one slides up to a new, higher, row, another one slides into its same place. This has the effect of incrementing an entry in $x_1,\ldots,x_n$ by some amount. So this is the reverse componentwise covering relation of the Tamari lattice.
\end{proof}

\begin{figure}[htbp]
\begin{center}
\scalebox{1}{
\begin{tikzpicture}[>=latex,line join=bevel,]
\node (node_5) at (35.500000bp,283.500000bp) [draw,draw=none] {\large{\mymk{\textcolor{red}{\textbf{321}}}}};
  \node (node_4) at (13.500000bp,216.500000bp) [draw,draw=none] {\large{\mymk{\textcolor{red}{\textbf{312}}}}};
  \node (node_3) at (58.500000bp,216.500000bp) [draw,draw=none] {\large{\mymk{\textcolor{red}{\textbf{231}}}}};
  \node (node_2) at (14.500000bp,149.500000bp) [draw,draw=none] {\large{\mymk{\textcolor{red}{\textbf{213}}}}};
  \node (node_1) at (35.500000bp,82.500000bp) [draw,draw=none] {\large{132}};
  \node (node_0) at (35.500000bp,15.500000bp) [draw,draw=none] {\large{\mymk{\textcolor{red}{\textbf{123}}}}};
  \draw [black,->] (node_4) ..controls (21.080000bp,239.900000bp) and (24.308000bp,249.430000bp)  .. (node_5);
  \draw [black,->] (node_1) ..controls (28.265000bp,105.900000bp) and (25.184000bp,115.430000bp)  .. (node_2);
  \draw [black,->] (node_1) ..controls (41.952000bp,120.530000bp) and (49.455000bp,163.590000bp)  .. (node_3);
  \draw [black,->] (node_3) ..controls (50.576000bp,239.900000bp) and (47.201000bp,249.430000bp)  .. (node_5);
  \draw [black,->] (node_0) ..controls (35.500000bp,38.729000bp) and (35.500000bp,48.018000bp)  .. (node_1);
  \draw [black,->] (node_2) ..controls (14.158000bp,172.730000bp) and (14.015000bp,182.020000bp)  .. (node_4);
\end{tikzpicture}
}
\scalebox{.6}{
\begin{tikzpicture}[>=latex,line join=bevel,]
\node (node_22) at (115.66bp,485.5bp) [draw,draw=none] {\Large{\mymk{\textcolor{red}{\textbf{4312}}}}};
  \node (node_23) at (115.66bp,538.5bp) [draw,draw=none] {\Large{\mymk{\textcolor{red}{\textbf{4321}}}}};
  \node (node_20) at (140.66bp,432.5bp) [draw,draw=none] {\Large{\mymk{\textcolor{red}{\textbf{4213}}}}};
  \node (node_21) at (203.66bp,432.5bp) [draw,draw=none] {\Large{\mymk{\textcolor{red}{\textbf{4231}}}}};
  \node (node_9) at (33.66bp,220.5bp) [draw,draw=none] {\Large{\mymk{\textcolor{red}{\textbf{2341}}}}};
  \node (node_8) at (96.66bp,220.5bp) [draw,draw=none] {\Large{\mymk{\textcolor{red}{\textbf{2314}}}}};
  \node (node_7) at (222.66bp,220.5bp) [draw,draw=none] {\Large{2143}};
  \node (node_6) at (222.66bp,167.5bp) [draw,draw=none] {\Large{\mymk{\textcolor{red}{\textbf{2134}}}}};
  \node (node_5) at (159.66bp,220.5bp) [draw,draw=none] {\Large{1432}};
  \node (node_4) at (159.66bp,167.5bp) [draw,draw=none] {\Large{1423}};
  \node (node_3) at (57.66bp,167.5bp) [draw,draw=none] {\Large{1342}};
  \node (node_2) at (133.66bp,114.5bp) [draw,draw=none] {\Large{1324}};
  \node (node_1) at (107.66bp,61.5bp) [draw,draw=none] {\Large{1243}};
  \node (node_0) at (107.66bp,8.5bp) [draw,draw=none] {\Large{\mymk{\textcolor{red}{\textbf{1234}}}}};
  \node (node_19) at (210.66bp,379.5bp) [draw,draw=none] {\Large{4132}};
  \node (node_18) at (261.66bp,326.5bp) [draw,draw=none] {\Large{\mymk{\textcolor{red}{\textbf{4123}}}}};
  \node (node_17) at (32.66bp,379.5bp) [draw,draw=none] {\Large{\mymk{\textcolor{red}{\textbf{3421}}}}};
  \node (node_16) at (33.66bp,326.5bp) [draw,draw=none] {\Large{\mymk{\textcolor{red}{\textbf{3412}}}}};
  \node (node_15) at (96.66bp,326.5bp) [draw,draw=none] {\Large{\mymk{\textcolor{red}{\textbf{3241}}}}};
  \node (node_14) at (159.66bp,326.5bp) [draw,draw=none] {\Large{\mymk{\textcolor{red}{\textbf{3214}}}}};
  \node (node_13) at (159.66bp,273.5bp) [draw,draw=none] {\Large{3142}};
  \node (node_12) at (222.66bp,273.5bp) [draw,draw=none] {\Large{\mymk{\textcolor{red}{\textbf{3124}}}}};
  \node (node_11) at (33.66bp,273.5bp) [draw,draw=none] {\Large{2431}};
  \node (node_10) at (96.66bp,273.5bp) [draw,draw=none] {\Large{2413}};
  \draw [black,->] (node_11) ..controls (52.541bp,289.78bp) and (67.706bp,302.06bp)  .. (node_15);
  \draw [black,->] (node_12) ..controls (234.0bp,289.33bp) and (242.65bp,300.65bp)  .. (node_18);
  \draw [black,->] (node_7) ..controls (203.78bp,236.78bp) and (188.61bp,249.06bp)  .. (node_13);
  \draw [black,->] (node_6) ..controls (222.66bp,182.81bp) and (222.66bp,193.03bp)  .. (node_7);
  \draw [black,->] (node_1) ..controls (115.07bp,77.031bp) and (120.52bp,87.72bp)  .. (node_2);
  \draw [black,->] (node_3) ..controls (50.822bp,183.03bp) and (45.792bp,193.72bp)  .. (node_9);
  \draw [black,->] (node_5) ..controls (120.13bp,237.5bp) and (85.879bp,251.36bp)  .. (node_11);
  \draw [black,->] (node_19) ..controls (189.58bp,395.86bp) and (172.5bp,408.3bp)  .. (node_20);
  \draw [black,->] (node_8) ..controls (96.66bp,235.81bp) and (96.66bp,246.03bp)  .. (node_10);
  \draw [black,->] (node_20) ..controls (133.54bp,448.03bp) and (128.3bp,458.72bp)  .. (node_22);
  \draw [black,->] (node_4) ..controls (159.66bp,182.81bp) and (159.66bp,193.03bp)  .. (node_5);
  \draw [black,->] (node_7) ..controls (222.66bp,235.81bp) and (222.66bp,246.03bp)  .. (node_12);
  \draw [black,->] (node_16) ..controls (46.007bp,343.75bp) and (56.228bp,357.92bp)  .. (63.66bp,371.0bp) .. controls (82.339bp,403.88bp) and (99.709bp,444.65bp)  .. (node_22);
  \draw [black,->] (node_10) ..controls (115.54bp,289.78bp) and (130.71bp,302.06bp)  .. (node_14);
  \draw [black,->] (node_22) ..controls (115.66bp,500.81bp) and (115.66bp,511.03bp)  .. (node_23);
  \draw [black,->] (node_17) ..controls (42.439bp,406.95bp) and (61.556bp,456.12bp)  .. (83.66bp,494.0bp) .. controls (89.291bp,503.65bp) and (96.613bp,513.74bp)  .. (node_23);
  \draw [black,->] (node_13) ..controls (173.0bp,290.4bp) and (183.78bp,304.44bp)  .. (190.66bp,318.0bp) .. controls (197.63bp,331.75bp) and (202.94bp,348.46bp)  .. (node_19);
  \draw [black,->] (node_5) ..controls (140.78bp,236.78bp) and (125.61bp,249.06bp)  .. (node_10);
  \draw [black,->] (node_2) ..controls (124.93bp,140.04bp) and (110.85bp,179.61bp)  .. (node_8);
  \draw [black,->] (node_19) ..controls (208.7bp,394.81bp) and (207.29bp,405.03bp)  .. (node_21);
  \draw [black,->] (node_10) ..controls (77.78bp,289.78bp) and (62.615bp,302.06bp)  .. (node_16);
  \draw [black,->] (node_15) ..controls (122.53bp,352.64bp) and (165.59bp,394.5bp)  .. (node_21);
  \draw [black,->] (node_3) ..controls (89.208bp,184.27bp) and (115.92bp,197.63bp)  .. (node_5);
  \draw [black,->] (node_16) ..controls (33.38bp,341.81bp) and (33.179bp,352.03bp)  .. (node_17);
  \draw [black,->] (node_2) ..controls (141.07bp,130.03bp) and (146.52bp,140.72bp)  .. (node_4);
  \draw [black,->] (node_21) ..controls (182.58bp,458.42bp) and (147.89bp,499.41bp)  .. (node_23);
  \draw [black,->] (node_3) ..controls (69.003bp,183.33bp) and (77.655bp,194.65bp)  .. (node_8);
  \draw [black,->] (node_5) ..controls (159.66bp,235.81bp) and (159.66bp,246.03bp)  .. (node_13);
  \draw [black,->] (node_14) ..controls (155.21bp,351.89bp) and (148.08bp,390.89bp)  .. (node_20);
  \draw [black,->] (node_13) ..controls (159.66bp,288.81bp) and (159.66bp,299.03bp)  .. (node_14);
  \draw [black,->] (node_13) ..controls (140.78bp,289.78bp) and (125.61bp,302.06bp)  .. (node_15);
  \draw [black,->] (node_0) ..controls (107.66bp,23.805bp) and (107.66bp,34.034bp)  .. (node_1);
  \draw [black,->] (node_4) ..controls (178.54bp,183.78bp) and (193.71bp,196.06bp)  .. (node_7);
  \draw [black,->] (node_18) ..controls (246.6bp,342.56bp) and (234.81bp,354.35bp)  .. (node_19);
  \draw [black,->] (node_11) ..controls (17.669bp,290.32bp) and (6.3382bp,303.75bp)  .. (1.6605bp,318.0bp) .. controls (-3.6109bp,334.06bp) and (6.7211bp,351.09bp)  .. (node_17);
  \draw [black,->] (node_1) ..controls (95.828bp,87.113bp) and (76.668bp,126.96bp)  .. (node_3);
  \draw [black,->] (node_2) ..controls (160.99bp,131.16bp) and (183.86bp,144.26bp)  .. (node_6);
  \draw [black,->] (node_12) ..controls (203.78bp,289.78bp) and (188.61bp,302.06bp)  .. (node_14);
  \draw [black,->] (node_9) ..controls (33.66bp,235.81bp) and (33.66bp,246.03bp)  .. (node_11);
\end{tikzpicture}
}
\end{center}
\caption{Left: $T_3^{\rm Perm}$, the poset of permutation magog triangles of order $3$; Right: $T_4^{\rm Perm}$. In both pictures, the $132$-avoiding permutations are circled; Theorem~\ref{thm:tamari} shows the subposet of $T_n^{\rm Perm}$ consisting of the $132$-avoiding permutations is the Tamari lattice of order $n$.}
\label{fig:Tpermmagog34}
\end{figure}
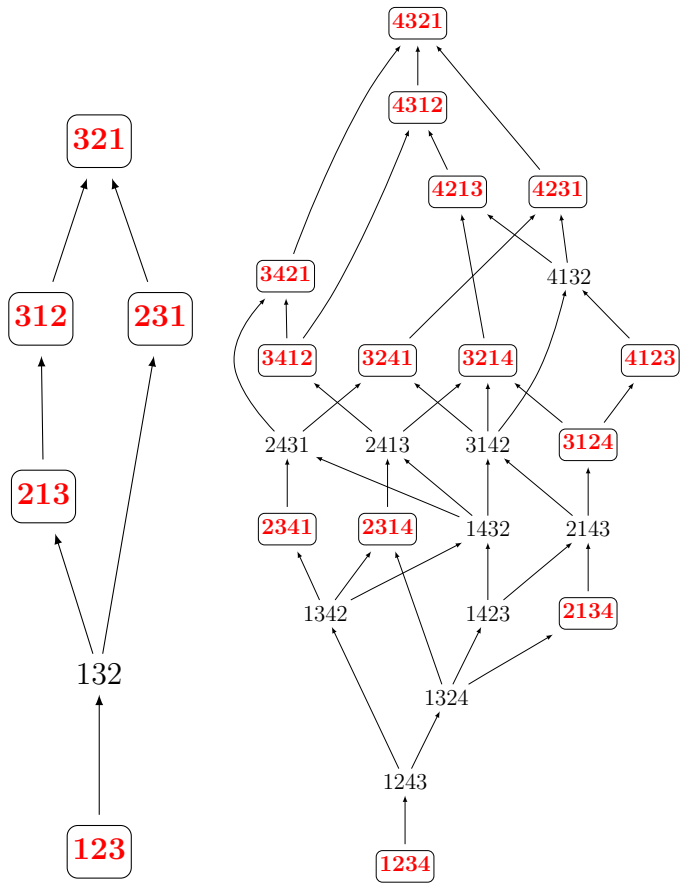

\begin{figure}
\begin{center}
\begin{tikzpicture}[>=latex,line join=bevel,]
\node (node_4) at (47.000000bp,174.500000bp) [draw,draw=none] {$123$};
  \node (node_3) at (75.000000bp,119.500000bp) [draw,draw=none] {$223$};
  \node (node_2) at (72.000000bp,64.500000bp) [draw,draw=none] {$323$};
  \node (node_1) at (19.000000bp,119.500000bp) [draw,draw=none] {$133$};
  \node (node_0) at (47.000000bp,9.500000bp) [draw,draw=none] {$333$};
  \draw [black,->] (node_1) ..controls (27.456000bp,136.510000bp) and (33.024000bp,147.050000bp)  .. (node_4);
  \draw [black,->] (node_0) ..controls (40.153000bp,36.908000bp) and (29.875000bp,76.554000bp)  .. (node_1);
  \draw [black,->] (node_0) ..controls (54.514000bp,26.429000bp) and (59.416000bp,36.822000bp)  .. (node_2);
  \draw [black,->] (node_2) ..controls (72.889000bp,81.198000bp) and (73.452000bp,91.155000bp)  .. (node_3);
  \draw [black,->] (node_3) ..controls (66.544000bp,136.510000bp) and (60.976000bp,147.050000bp)  .. (node_4);
\end{tikzpicture}
\hspace{1in}
\begin{tikzpicture}[>=latex,line join=bevel,]
\node (node_4) at (47.000000bp,9.500000bp) [draw,draw=none] {$333$};
  \node (node_3) at (47.000000bp,64.500000bp) [draw,draw=none] {$233$};
  \node (node_2) at (75.000000bp,119.500000bp) [draw,draw=none] {$223$};
  \node (node_1) at (19.000000bp,119.500000bp) [draw,draw=none] {$133$};
  \node (node_0) at (47.000000bp,174.500000bp) [draw,draw=none] {$123$};
  \draw [black,->] (node_3) ..controls (55.456000bp,81.507000bp) and (61.024000bp,92.046000bp)  .. (node_2);
  \draw [black,->] (node_4) ..controls (47.000000bp,26.198000bp) and (47.000000bp,36.155000bp)  .. (node_3);
  \draw [black,->] (node_2) ..controls (66.544000bp,136.510000bp) and (60.976000bp,147.050000bp)  .. (node_0);
  \draw [black,->] (node_1) ..controls (27.456000bp,136.510000bp) and (33.024000bp,147.050000bp)  .. (node_0);
  \draw [black,->] (node_3) ..controls (38.544000bp,81.507000bp) and (32.976000bp,92.046000bp)  .. (node_1);
\end{tikzpicture}
\end{center}
\caption{Left: The poset $Tam_3$ of Definition~\ref{def:tamari}; Right: The poset $Cat_3$ of Definition~\ref{def:Catdistr}}
\label{fig:Cat_compare}
\end{figure}
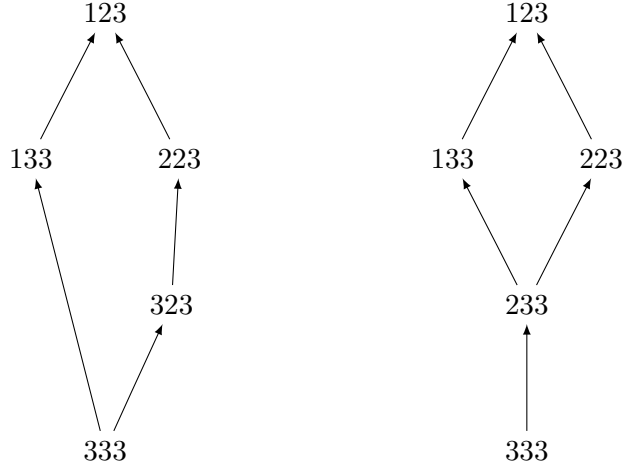

We now show $T_n^{\rm Perm}$ contains another Catalan poset. The \emph{Catalan distributive lattice} of order $n$ is the natural containment partial order on several Catalan objects, such as \emph{Dyck paths}, partitions inside a staircase, and order ideals in the \emph{Type $A$ positive root poset}. We give the following equivalent definition, written in a similar form to Definition~\ref{def:tamari} above. See Figure~\ref{fig:Cat_compare}. 

\begin{definition}
\label{def:Catdistr}
Let the Catalan distributive lattice of order $n$, which we denote as $Cat_n$, be defined as the set of sequences of $n$ positive integers $x_1, x_2,\ldots, x_n$, such that $i\leq x_i \leq n$ and if $i\leq j$ then $x_i\leq x_j$, ordered by reverse componentwise comparison.
\end{definition}

Note that by subtracting $i$ from $x_i$, the \emph{area sequence} of the corresponding Dyck path is obtained.

\begin{theorem}
\label{thm:Catdistr}
$T_n^{\rm Perm}$ contains a subposet which is isomorphic to $Cat_n$. In particular, the subposet of $T_n^{\rm Perm}$ consisting of the $213$-avoiding permutations is isomorphic to $Cat_n$.
\end{theorem}

\begin{proof}
Let $b$ be a permutation TSSCPP boolean triangle. As in the proof of Theorem~\ref{thm:tamari}, define a sequence of positive integers $x_1,x_2,\ldots,x_n$ where $x_i = i + \Sigma_{n-i}$, and we consider $\Sigma_0=0$. This is an injection from the set of permutation TSSCPP boolean triangles of order $n$ to the set of sequences of integers $x_1,x_2,\ldots,x_n$ with $i\leq x_i\leq n$, since the sum of row $n-i$ of $b$ can be at most $n-i$.

We wish to show that if the permutation associated with $b$ is $213$-avoiding, then $x_1,x_2,\ldots,x_n$ satisfies the inequalities in Definition~\ref{def:Catdistr}. As discussed above, $i\leq x_i\leq n$. Now we wish to show that the $213$-avoiding condition on the permutation corresponding to $b$ corresponds to the property that if $i\leq j$ then $x_i\leq x_j$.
In the bijection of Theorem~\ref{thm:tsscppsn}, the permutation $213$ corresponds to the TSSCPP boolean triangle 
\scalebox{.8}{$\begin{array}{ccc}
& 0 &\\
1 & &1
\end{array}$}.
More generally, we wish to show that $b$ corresponds to a permutation that is $213$-avoiding if and only if $b$ has the following property:

\emph{Property $(***)$}: If $b_{i,j}=0$, then $b_{i',j}=0$ for all $i'>i$.

%Let $b$ be the boolean triangle of a $213$-avoiding permutation. We claim that if $b_{i,j}=0$, then $b_{i',j}=0$ for all $i'>i$. 
Suppose to the contrary that $b_{i,j}=0$ and $b_{i+1,j}=1$. 
Then in the corresponding monotone triangle, $a_{i,j}=a_{i+1,j}= a_{i+2,j-1}$. Also, $\sigma(i+2)>a_{i+2,j-1}$ since $b_{i+1,j}=1$ and the ones are left-justified, while $\sigma(i+1)<a_{i+1,j}$ since $b_{i,j}=0$ and the new value in a row of the monotone triangle $a$ sits horizontally between the rightmost one and leftmost zero in the previous row of the corresponding boolean triangle. Finally, since the value $a_{i,j}$ is in row $i$ of the monotone triangle, $a_{i,j}=\sigma(k)$ for some $k\leq i$. That is, $a_{i,j}$ must be the new value in $a$ either in row $i$ or in some previous row. Putting these facts together, we have $\sigma(i+1)<\sigma(k)<\sigma(i+2)$ and $k<i+1<i+2$, so $\sigma(k), \sigma(i+1), \sigma(i+2)$ is a $213$-pattern in the permutation $\sigma$. Therefore, we have proven that $(***)$ holds whenever $b$ corresponds to a $213$-avoiding permutation. Furthermore, any $213$ pattern in the permutation corresponding to $b$ will produce a violation of $(***)$ in the same way as discussed above. So $(***)$ holds if and only if the permutation corresponding to $b$ contains a $213$ pattern. 
 
Now suppose $i< j$ and $x_i>x_j$. Since $i< j$, this means $n-i> n-j$. Now $x_i>x_j$ implies that in the boolean triangle $b$, $i + \Sigma_{n-i} > j + \Sigma_{n-j}$. That is, $\Sigma_{n-i} > j-i+\Sigma_{n-j}$.
%$i$ plus the sum of row $n-i$ is greater than $j$ plus the sum of row $n-j$. That is,  the sum of row $n-i$ is greater than $j-i$ plus the sum of row $n-j$. 
Thus, there must be a $k$ such that $b_{n-j,k}=0$ and $b_{n-i,k}=1$, since the entries equal to one in a permutation boolean triangle must be left-justified. Thus by the previous paragraph, there must be a $213$-pattern in the corresponding permutation. Therefore, the $213$-avoiding condition on the permutation corresponding to $b$ is equivalent to the property that if $i\leq j$ then $x_i\leq x_j$. 

Finally, we show that the reverse componentwise comparison covering relation of the Catalan distributive sequences $x_1,\ldots,x_n$ equals the magog partial order on $213$-avoiding permutation TSSCPP boolean triangles, whose covering relations are given in Lemma \ref{lem:magogboolcovers}. 
The reverse componentwise comparison covering relation on the sequences $x_1,\ldots,x_n$ corresponds to the rightmost one in some row of $b$ turning into a zero. 
This one must also have a one to its northwest and a zero to its southeast (if these entries exist in $b$) so that the $213$-avoiding condition is satisfied. 
Thus, removing this one is equivalent to sliding it to the southeast until it leaves the matrix, which is a composition of the magog covering relations of Lemma~\ref{lem:magogboolcovers}. 
Also, each slide of this one a single entry to the southeast results in another boolean triangle, since this sliding will not violate the defining inequality of Definition~\ref{def:bool}. 
Thus, in the case of $213$-avoiding boolean triangles, the magog covering relations correspond to the Catalan distributive order.
\end{proof}

\begin{figure}[hbtp]
\begin{center}
\scalebox{1}{
\begin{tikzpicture}[>=latex,line join=bevel,]
\node (node_5) at (35.500000bp,283.500000bp) [draw,draw=none] {\large{\mymk{\textcolor{magenta}{\textbf{321}}}}};
  \node (node_4) at (13.500000bp,216.500000bp) [draw,draw=none] {\large{\mymk{\textcolor{magenta}{\textbf{312}}}}};
  \node (node_3) at (58.500000bp,216.500000bp) [draw,draw=none] {\large{\mymk{\textcolor{magenta}{\textbf{231}}}}};
  \node (node_2) at (14.500000bp,149.500000bp) [draw,draw=none] {\large{213}};
  \node (node_1) at (35.500000bp,82.500000bp) [draw,draw=none] {\large{\mymk{\textcolor{magenta}{\textbf{132}}}}};
  \node (node_0) at (35.500000bp,15.500000bp) [draw,draw=none] {\large{\mymk{\textcolor{magenta}{\textbf{123}}}}};
  \draw [black,->] (node_4) ..controls (21.080000bp,239.900000bp) and (24.308000bp,249.430000bp)  .. (node_5);
  \draw [black,->] (node_1) ..controls (28.265000bp,105.900000bp) and (25.184000bp,115.430000bp)  .. (node_2);
  \draw [black,->] (node_1) ..controls (41.952000bp,120.530000bp) and (49.455000bp,163.590000bp)  .. (node_3);
  \draw [black,->] (node_3) ..controls (50.576000bp,239.900000bp) and (47.201000bp,249.430000bp)  .. (node_5);
  \draw [black,->] (node_0) ..controls (35.500000bp,38.729000bp) and (35.500000bp,48.018000bp)  .. (node_1);
  \draw [black,->] (node_2) ..controls (14.158000bp,172.730000bp) and (14.015000bp,182.020000bp)  .. (node_4);
\end{tikzpicture}
}
\scalebox{.6}{
\begin{tikzpicture}[>=latex,line join=bevel,]
\node (node_22) at (115.66bp,485.5bp) [draw,draw=none] {\Large{\mymk{\textcolor{magenta}{\textbf{4312}}}}};
  \node (node_23) at (115.66bp,538.5bp) [draw,draw=none] {\Large{\mymk{\textcolor{magenta}{\textbf{4321}}}}};
  \node (node_20) at (140.66bp,432.5bp) [draw,draw=none] {\Large{4213}};
  \node (node_21) at (203.66bp,432.5bp) [draw,draw=none] {\Large{\mymk{\textcolor{magenta}{\textbf{4231}}}}};
  \node (node_9) at (33.66bp,220.5bp) [draw,draw=none] {\Large{\mymk{\textcolor{magenta}{\textbf{2341}}}}};
  \node (node_8) at (96.66bp,220.5bp) [draw,draw=none] {\Large{2314}};
  \node (node_7) at (222.66bp,220.5bp) [draw,draw=none] {\Large{2143}};
  \node (node_6) at (222.66bp,167.5bp) [draw,draw=none] {\Large{2134}};
  \node (node_5) at (159.66bp,220.5bp) [draw,draw=none] {\Large{\mymk{\textcolor{magenta}{\textbf{1432}}}}};
  \node (node_4) at (159.66bp,167.5bp) [draw,draw=none] {\Large{\mymk{\textcolor{magenta}{\textbf{1423}}}}};
  \node (node_3) at (57.66bp,167.5bp) [draw,draw=none] {\Large{\mymk{\textcolor{magenta}{\textbf{1342}}}}};
  \node (node_2) at (133.66bp,114.5bp) [draw,draw=none] {\Large{1324}};
  \node (node_1) at (107.66bp,61.5bp) [draw,draw=none] {\Large{\mymk{\textcolor{magenta}{\textbf{1243}}}}};
  \node (node_0) at (107.66bp,8.5bp) [draw,draw=none] {\Large{\mymk{\textcolor{magenta}{\textbf{1234}}}}};
  \node (node_19) at (210.66bp,379.5bp) [draw,draw=none] {\Large{\mymk{\textcolor{magenta}{\textbf{4132}}}}};
  \node (node_18) at (261.66bp,326.5bp) [draw,draw=none] {\Large{\mymk{\textcolor{magenta}{\textbf{4123}}}}};
  \node (node_17) at (32.66bp,379.5bp) [draw,draw=none] {\Large{\mymk{\textcolor{magenta}{\textbf{3421}}}}};
  \node (node_16) at (33.66bp,326.5bp) [draw,draw=none] {\Large{\mymk{\textcolor{magenta}{\textbf{3412}}}}};
  \node (node_15) at (96.66bp,326.5bp) [draw,draw=none] {\Large{3241}};
  \node (node_14) at (159.66bp,326.5bp) [draw,draw=none] {\Large{3214}};
  \node (node_13) at (159.66bp,273.5bp) [draw,draw=none] {\Large{3142}};
  \node (node_12) at (222.66bp,273.5bp) [draw,draw=none] {\Large{3124}};
  \node (node_11) at (33.66bp,273.5bp) [draw,draw=none] {\Large{\mymk{\textcolor{magenta}{\textbf{2431}}}}};
  \node (node_10) at (96.66bp,273.5bp) [draw,draw=none] {\Large{2413}};
  \draw [black,->] (node_11) ..controls (52.541bp,289.78bp) and (67.706bp,302.06bp)  .. (node_15);
  \draw [black,->] (node_12) ..controls (234.0bp,289.33bp) and (242.65bp,300.65bp)  .. (node_18);
  \draw [black,->] (node_7) ..controls (203.78bp,236.78bp) and (188.61bp,249.06bp)  .. (node_13);
  \draw [black,->] (node_6) ..controls (222.66bp,182.81bp) and (222.66bp,193.03bp)  .. (node_7);
  \draw [black,->] (node_1) ..controls (115.07bp,77.031bp) and (120.52bp,87.72bp)  .. (node_2);
  \draw [black,->] (node_3) ..controls (50.822bp,183.03bp) and (45.792bp,193.72bp)  .. (node_9);
  \draw [black,->] (node_5) ..controls (120.13bp,237.5bp) and (85.879bp,251.36bp)  .. (node_11);
  \draw [black,->] (node_19) ..controls (189.58bp,395.86bp) and (172.5bp,408.3bp)  .. (node_20);
  \draw [black,->] (node_8) ..controls (96.66bp,235.81bp) and (96.66bp,246.03bp)  .. (node_10);
  \draw [black,->] (node_20) ..controls (133.54bp,448.03bp) and (128.3bp,458.72bp)  .. (node_22);
  \draw [black,->] (node_4) ..controls (159.66bp,182.81bp) and (159.66bp,193.03bp)  .. (node_5);
  \draw [black,->] (node_7) ..controls (222.66bp,235.81bp) and (222.66bp,246.03bp)  .. (node_12);
  \draw [black,->] (node_16) ..controls (46.007bp,343.75bp) and (56.228bp,357.92bp)  .. (63.66bp,371.0bp) .. controls (82.339bp,403.88bp) and (99.709bp,444.65bp)  .. (node_22);
  \draw [black,->] (node_10) ..controls (115.54bp,289.78bp) and (130.71bp,302.06bp)  .. (node_14);
  \draw [black,->] (node_22) ..controls (115.66bp,500.81bp) and (115.66bp,511.03bp)  .. (node_23);
  \draw [black,->] (node_17) ..controls (42.439bp,406.95bp) and (61.556bp,456.12bp)  .. (83.66bp,494.0bp) .. controls (89.291bp,503.65bp) and (96.613bp,513.74bp)  .. (node_23);
  \draw [black,->] (node_13) ..controls (173.0bp,290.4bp) and (183.78bp,304.44bp)  .. (190.66bp,318.0bp) .. controls (197.63bp,331.75bp) and (202.94bp,348.46bp)  .. (node_19);
  \draw [black,->] (node_5) ..controls (140.78bp,236.78bp) and (125.61bp,249.06bp)  .. (node_10);
  \draw [black,->] (node_2) ..controls (124.93bp,140.04bp) and (110.85bp,179.61bp)  .. (node_8);
  \draw [black,->] (node_19) ..controls (208.7bp,394.81bp) and (207.29bp,405.03bp)  .. (node_21);
  \draw [black,->] (node_10) ..controls (77.78bp,289.78bp) and (62.615bp,302.06bp)  .. (node_16);
  \draw [black,->] (node_15) ..controls (122.53bp,352.64bp) and (165.59bp,394.5bp)  .. (node_21);
  \draw [black,->] (node_3) ..controls (89.208bp,184.27bp) and (115.92bp,197.63bp)  .. (node_5);
  \draw [black,->] (node_16) ..controls (33.38bp,341.81bp) and (33.179bp,352.03bp)  .. (node_17);
  \draw [black,->] (node_2) ..controls (141.07bp,130.03bp) and (146.52bp,140.72bp)  .. (node_4);
  \draw [black,->] (node_21) ..controls (182.58bp,458.42bp) and (147.89bp,499.41bp)  .. (node_23);
  \draw [black,->] (node_3) ..controls (69.003bp,183.33bp) and (77.655bp,194.65bp)  .. (node_8);
  \draw [black,->] (node_5) ..controls (159.66bp,235.81bp) and (159.66bp,246.03bp)  .. (node_13);
  \draw [black,->] (node_14) ..controls (155.21bp,351.89bp) and (148.08bp,390.89bp)  .. (node_20);
  \draw [black,->] (node_13) ..controls (159.66bp,288.81bp) and (159.66bp,299.03bp)  .. (node_14);
  \draw [black,->] (node_13) ..controls (140.78bp,289.78bp) and (125.61bp,302.06bp)  .. (node_15);
  \draw [black,->] (node_0) ..controls (107.66bp,23.805bp) and (107.66bp,34.034bp)  .. (node_1);
  \draw [black,->] (node_4) ..controls (178.54bp,183.78bp) and (193.71bp,196.06bp)  .. (node_7);
  \draw [black,->] (node_18) ..controls (246.6bp,342.56bp) and (234.81bp,354.35bp)  .. (node_19);
  \draw [black,->] (node_11) ..controls (17.669bp,290.32bp) and (6.3382bp,303.75bp)  .. (1.6605bp,318.0bp) .. controls (-3.6109bp,334.06bp) and (6.7211bp,351.09bp)  .. (node_17);
  \draw [black,->] (node_1) ..controls (95.828bp,87.113bp) and (76.668bp,126.96bp)  .. (node_3);
  \draw [black,->] (node_2) ..controls (160.99bp,131.16bp) and (183.86bp,144.26bp)  .. (node_6);
  \draw [black,->] (node_12) ..controls (203.78bp,289.78bp) and (188.61bp,302.06bp)  .. (node_14);
  \draw [black,->] (node_9) ..controls (33.66bp,235.81bp) and (33.66bp,246.03bp)  .. (node_11);
\end{tikzpicture}
}
\end{center}
\caption{Left: $T_3^{\rm Perm}$, the poset of permutation magog triangles of order $3$; Right: $T_4^{\rm Perm}$. In both pictures, the $213$-avoiding permutations are circled; Theorem~\ref{thm:Catdistr} shows the subposet of $T_n^{\rm Perm}$ consisting of the $213$-avoiding permutations is the Catalan distributive lattice of order $n$.}
\label{fig:Catdist34}
\end{figure}
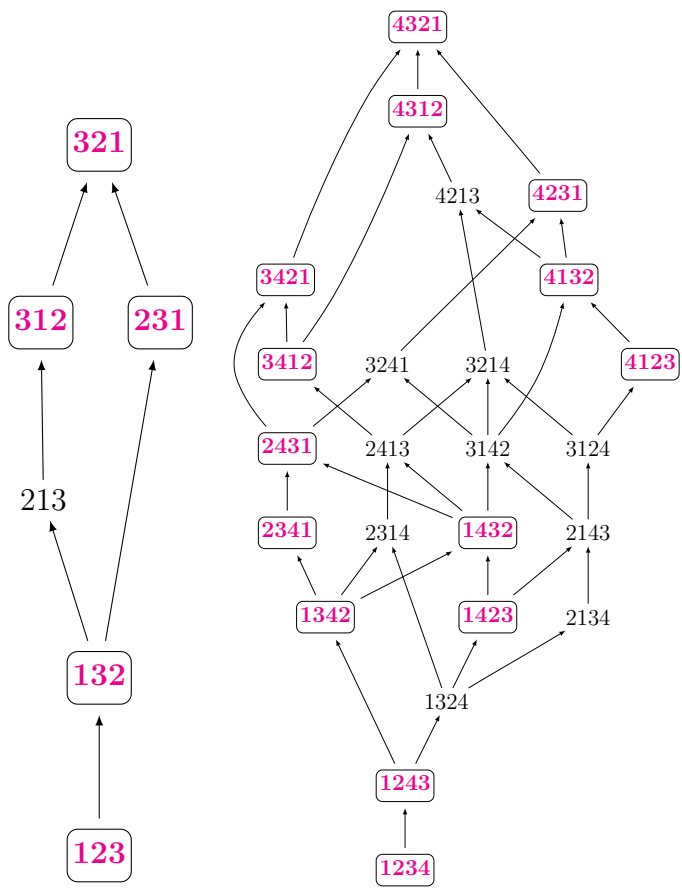

\begin{remark}
A natural question is whether $T_n^{\rm Perm}$ restricted to the permutations which avoid the pattern $123$, $231$, $312$, or $321$ yields a known Catalan poset. We have calculated these posets in Sage~\cite{sage} and have noted that for $n=4$, none of these posets are ranked and none are lattices. So these posets are not $Tam_n$ or $Cat_n$, since both of these are lattices; it is unknown whether they have some interpretation in terms of other Catalan objects. Note also that in~\cite{312ASM}, A.\ Ayyer, R.\ Cori, and D.\ Gouyou-Beauchamps showed the set of alternating sign matrices with \emph{gapless} monotone triangles contains exactly the $312$-avoiding permutations as its permutation subset. In both \cite{312ASM} and \cite{STRIKERPOSET}, this subset of monotone triangles was studied and shown to be a subset in which ASM and TSSCPP are in natural bijection. In \cite{STRIKERPOSET}, we additionally studied the poset structure of gapless monotone triangles, which we denoted $J(\mathrm{ASM}\cap \mathrm{TSSCPP})$ since we found it to be a distributive lattice related to $A_n$ and $T_n$ in the family of \emph{tetrahedral posets}.
\end{remark}

\subsection{The TSSCPP boolean poset and Bruhat order}
\label{subsec:TBool}
We now define a new TSSCPP poset using boolean triangles. See Figure~\ref{fig:TBool3} for an example in the case $n=3$.

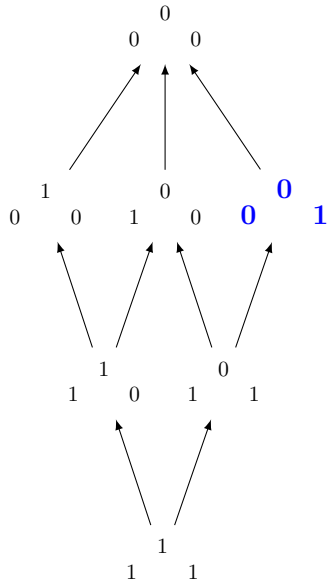
\begin{figure}[htbp]
\begin{center}
\begin{tikzpicture}[>=latex,line join=bevel,]
\node (node_6) at (57.500000bp,15.500000bp) [draw,draw=none] {\scalebox{.75}{$\begin{array}{ccc}
& 1 &\\
1 & &1
\end{array}$}};
  \node (node_5) at (35.500000bp,82.500000bp) [draw,draw=none] {\scalebox{.75}{$\begin{array}{ccc}
& 1 &\\
1 & &0
\end{array}$}};
  \node (node_4) at (80.500000bp,82.500000bp) [draw,draw=none] {\scalebox{.75}{$\begin{array}{ccc}
& 0 &\\
1 & &1
\end{array}$}};
  \node (node_3) at (58.500000bp,149.500000bp) [draw,draw=none] {\scalebox{.75}{$\begin{array}{ccc}
& 0 &\\
1 & &0
\end{array}$}};
  \node (node_2) at (13.500000bp,149.500000bp) [draw,draw=none] {\scalebox{.75}{$\begin{array}{ccc}
& 1 &\\
0 & &0
\end{array}$}};
  \node (node_1) at (103.500000bp,149.500000bp) [draw,draw=none] {\scalebox{.75}{$\begin{array}{ccc}
& \Large{\textcolor{blue}{\textbf{0}}} &\\
\Large{\textcolor{blue}{\textbf{0}}} & &\Large{\textcolor{blue}{\textbf{1}}}
\end{array}$}};
  \node (node_0) at (58.500000bp,216.500000bp) [draw,draw=none] {\scalebox{.75}{$\begin{array}{ccc}
& 0 &\\
0 & &0
\end{array}$}};
  \draw [black,->] (node_3) ..controls (58.500000bp,172.730000bp) and (58.500000bp,182.020000bp)  .. (node_0);
  \draw [black,->] (node_5) ..controls (27.920000bp,105.900000bp) and (24.692000bp,115.430000bp)  .. (node_2);
  \draw [black,->] (node_2) ..controls (29.292000bp,173.310000bp) and (36.328000bp,183.470000bp)  .. (node_0);
  \draw [black,->] (node_5) ..controls (43.424000bp,105.900000bp) and (46.799000bp,115.430000bp)  .. (node_3);
  \draw [black,->] (node_6) ..controls (49.920000bp,38.895000bp) and (46.692000bp,48.432000bp)  .. (node_5);
  \draw [black,->] (node_6) ..controls (65.424000bp,38.895000bp) and (68.799000bp,48.432000bp)  .. (node_4);
  \draw [black,->] (node_1) ..controls (87.708000bp,173.310000bp) and (80.672000bp,183.470000bp)  .. (node_0);
  \draw [black,->] (node_4) ..controls (88.424000bp,105.900000bp) and (91.799000bp,115.430000bp)  .. (node_1);
  \draw [black,->] (node_4) ..controls (72.920000bp,105.900000bp) and (69.692000bp,115.430000bp)  .. (node_3);
\end{tikzpicture}
\end{center}
\caption{The poset $TBool_3$ of TSSCPP boolean triangles of order $3$, partially ordered by reverse componentwise comparison. Note that the bold boolean triangle is the only non-permutation TSSCPP; the induced subposet obtained by removing this element is $[2]\times [3]$.}
\label{fig:TBool3}
\end{figure}

\begin{definition} \rm
Define the \emph{boolean partial order} $TBool_n$ (or the TSSCPP boolean poset of order $n$) by reverse componentwise comparison of the TSSCPP boolean triangles of order $n$. 
\end{definition}

Note that $TBool_n$ is the induced subposet of the Boolean lattice on $\binom{n}{2}$ elements given by only taking the elements corresponding to TSSCPP boolean triangles. 

\begin{proposition}
\label{prop:notlattice}
$TBool_n$ is a lattice for $n\leq 3$, but for $n\geq 4$ it is not a lattice.
\end{proposition}

\begin{proof}
We have computed $TBool_n$ in Sage~\cite{sage} for $n\leq 5$. $TBool_3$ is the (non-distributive) lattice in Figure~\ref{fig:TBool3}. We claim $TBool_4$ is not a lattice. Consider the following elements of $TBool_4$:

\begin{center}
$b_1 =$ \scalebox{.75}{$\begin{array}{ccccc}
& & 0 &\\
& 0 & &0\\
0 & & 0 & & 1
\end{array}$}
and
$b_2 =$ \scalebox{.75}{$\begin{array}{ccccc}
& & 0 &\\
& 0 & &1\\
0 & & 0 & & 0
\end{array}$}.
\end{center}

\noindent
The array \scalebox{.75}{$\begin{array}{ccccc}
& & 0 &\\
& 0 & &1\\
0 & & 0 & & 1
\end{array}$}
is not in $TBool_4$, since it violates the defining condition of Definition~\ref{def:bool}. So there is no rank $2$ element of $TBool_4$ greater than both $b_1$ and $b_2$. There are two rank $3$ elements of $TBool_4$ greater than $b_1$ and $b_2$:

\begin{center}
\scalebox{.75}{$\begin{array}{ccccc}
& & 0 &\\
& 1 & &1\\
0 & & 0 & & 1
\end{array}$}
and
\scalebox{.75}{$\begin{array}{ccccc}
& & 0 &\\
& 0 & &1\\
0 & & 1 & & 1
\end{array}$}. 
\end{center}

\noindent
So $b_1$ and $b_2$ violate the lattice condition, since they do not have a join (least upper bound).

For any $n>1$, $TBool_{n-1}$ sits as an interval inside $TBool_n$. Thus if $TBool_{n-1}$ is not a lattice, then neither is $TBool_n$.
\end{proof}

If we further restrict this order to permutation TSSCPP, we show, in Corollary~\ref{cor:TBool}, that the poset formed is an exceptionally nice distributive lattice which sits between the weak and strong Bruhat orders; see Figure~\ref{fig:snorders}. 

\begin{definition}
\label{def:permboolposet}
Let the \emph{TSSCPP boolean permutation poset}, denoted $TBool_n^{\rm Perm}$, on permutations of $n$ be given by the reverse componentwise comparison of the corresponding TSSCPP boolean triangles. That is, $TBool_n^{\rm Perm}$ is the induced subposet of $TBool_n$ consisting of the \emph{permutation} TSSCPP boolean triangles.
\end{definition}

We will also need the following definition; see~\cite{BjornerBrenti}.

\begin{definition}
\label{def:bruhat}
The \emph{weak order} on the symmetric group $S_n$ is the partial order on permutations of $n$ whose covering relations are given as: $\pi$ covers $\sigma$ if they differ by an adjacent transposition $(i, i+1)$ and there is an inversion between $i$ and $i+1$ in $\pi$. The \emph{strong Bruhat order} on $S_n$ is the partial order whose covering relations are given by: $\pi$ covers $\sigma$ if they differ by a transposition $(i, j)$ and $\pi$ has one more inversion than $\sigma$.
\end{definition}

As a corollary of Theorem~\ref{thm:tsscppsn}, we have the following; see Figure~\ref{fig:snorders}.
\begin{corollary}
\label{cor:TBool}
$TBool_n^{\rm Perm}$ is isomorphic to $[2]\times[3]\times\cdots\times[n]$, that is, the product of chains with $2, 3, \ldots, n$ elements. Thus, this
 is a partial order on permutations which sits between the weak and strong Bruhat orders on the symmetric group. That is, it contains all of the ordering relations of the weak order plus some of the additional relations of the strong order.
\end{corollary}
\begin{proof}
In the proof of Theorem~\ref{thm:tsscppsn} we noted that a zero in row $i$ of a permutation TSSCPP boolean triangle represents an inversion between $\sigma(i+1)$ and some $\sigma(k)$ with $k\leq i$ in the one-line notation of the corresponding permutation. So the number of zeros in row $i$ of the permutation TSSCPP boolean triangle equals the number of $k\leq i$ such that  $\sigma(k)>\sigma(i+1)$. As we noted in Proposition~\ref{prop:nfact}, the number of zeros in a row of a permutation TSSCPP boolean triangle can be chosen independently, thus $TBool_n^{\rm Perm}$ is isomorphic to the product of chains $[2]\times[3]\times\cdots\times[n]$; the order ideal composed of $\ell$ elements in the chain $[i]$ corresponds to row $i-1$ of the boolean triangle containing $\ell$ zeros. 

It is known that the product of chains poset $[2]\times[3]\times\cdots\times[n]$ sits between the weak and strong Bruhat orders (see, for example, \cite[p.\ 182]{StanleyLefschetz}, attributed to Gansner). We prove this directly in our context by showing the covering relations in $TBool_n^{\rm Perm}$ (1) contain the covering relations of the weak order and (2) are contained in the covering relations of the strong order.
To show (1), we note that swapping $\sigma(i)=k$ and $\sigma(j)=k+1$, $i<j$, creates a single new inversion at $\sigma(j)$; this adds a zero to row $j-1$ of the boolean triangle and leaves all other entries the same.  Thus, the covering relations of the weak order are also covering relations in $TBool_n^{\rm Perm}$. 

To show (2), we wish to show that changing a one into a zero in row $i$ of a permutation TSSCPP boolean triangle is a transposition on the corresponding permutation. Recall that the ones in a permutation TSSCPP boolean triangle are left-justified. Thus, increasing the number of zeros in row $i$ can only happen by changing the last one in the row to a zero. 
In particular, suppose two permutation boolean triangles $b$ and $b'$ are equal in all entries, except $b_{i,j}=1$ and $b'_{i,j}=0$. Thus $b'$ covers $b$ in $TBool_n^{\rm Perm}$. Let $a$ and $a'$ be the monotone triangles in bijection via Theorem~\ref{thm:tsscppsn} with $b$ and $b'$, respectively. Now $a$ and $a'$ are equal for rows $i+1$ through $n$. But $a'_{i,j} = a_{i+1,j}$ whereas $a_{i,j}=a_{i+1,j-1}$. Also, there is a (possibly empty) path above $a'_{i,j}$ which now equals the value $a_{i+1,j}$ rather than $a_{i+1,j-1}$. Now in $a$, the value $a_{i+1,j-1}$ first appears in some row $k$ for $k\leq i$. So in $a'$, the value $a_{i+1,j}$ first appears now in row $k$. This means in the corresponding permutations $\sigma$ and $\sigma'$, $\sigma(k)=a_{i+1,j-1}$ and $\sigma(i+1)=a_{i+1,j}$ whereas $\sigma'(k) = a_{i+1,j}$ and $\sigma'(i+1)=a_{i+1,j-1}$, and $\sigma$ and $\sigma'$ agree in all other positions. This is a transposition of $a_{i+1,j-1}$ and $a_{i+1,j}$, thus all covering relations in $TBool_n^{\rm Perm}$ are covering relations in the strong Bruhat order on permutations of $n$.
\end{proof}

\begin{figure}[hbtp]
\begin{center}
\begin{tikzpicture}[>=latex,line join=bevel,]
\node (node_5) at (45.000000bp,167.500000bp) [draw,draw=none] {$321$};
  \node (node_4) at (18.000000bp,114.500000bp) [draw,draw=none] {$312$};
  \node (node_3) at (72.000000bp,114.500000bp) [draw,draw=none] {$231$};
  \node (node_2) at (72.000000bp,61.500000bp) [draw,draw=none] {$213$};
  \node (node_1) at (18.000000bp,61.500000bp) [draw,draw=none] {$132$};
  \node (node_0) at (45.000000bp,8.500000bp) [draw,draw=none] {$123$};
  \draw [black,->] (node_1) ..controls (18.000000bp,76.805000bp) and (18.000000bp,87.034000bp)  .. (node_4);
  \draw [black,->] (node_2) ..controls (72.000000bp,76.805000bp) and (72.000000bp,87.034000bp)  .. (node_3);
  \draw [black,->] (node_0) ..controls (52.733000bp,24.106000bp) and (58.474000bp,34.950000bp)  .. (node_2);
  \draw [black,->] (node_4) ..controls (25.733000bp,130.110000bp) and (31.474000bp,140.950000bp)  .. (node_5);
  \draw [black,->] (node_3) ..controls (64.267000bp,130.110000bp) and (58.526000bp,140.950000bp)  .. (node_5);
  \draw [black,->] (node_0) ..controls (37.267000bp,24.106000bp) and (31.526000bp,34.950000bp)  .. (node_1);
\end{tikzpicture}
\hspace{.3in}
\begin{tikzpicture}[>=latex,line join=bevel,]
\node (node_5) at (45.000000bp,167.500000bp) [draw,draw=none] {$321$};
  \node (node_4) at (18.000000bp,114.500000bp) [draw,draw=none] {$312$};
  \node (node_3) at (72.000000bp,114.500000bp) [draw,draw=none] {$231$};
  \node (node_2) at (72.000000bp,61.500000bp) [draw,draw=none] {$213$};
  \node (node_1) at (18.000000bp,61.500000bp) [draw,draw=none] {$132$};
  \node (node_0) at (45.000000bp,8.500000bp) [draw,draw=none] {$123$};
  \draw [black,->] (node_1) ..controls (18.000000bp,76.805000bp) and (18.000000bp,87.034000bp)  .. (node_4);
  \draw [black,->] (node_2) ..controls (72.000000bp,76.805000bp) and (72.000000bp,87.034000bp)  .. (node_3);
  \draw [black,->] (node_0) ..controls (52.733000bp,24.106000bp) and (58.474000bp,34.950000bp)  .. (node_2);
  \draw [black,->] (node_4) ..controls (25.733000bp,130.110000bp) and (31.474000bp,140.950000bp)  .. (node_5);
  \draw [black,->] (node_1) ..controls (34.024000bp,77.634000bp) and (46.679000bp,89.586000bp)  .. (node_3);
  \draw [black,->] (node_3) ..controls (64.267000bp,130.110000bp) and (58.526000bp,140.950000bp)  .. (node_5);
  \draw [black,->] (node_0) ..controls (37.267000bp,24.106000bp) and (31.526000bp,34.950000bp)  .. (node_1);
\end{tikzpicture}
\hspace{.3in}
\begin{tikzpicture}[>=latex,line join=bevel,]
\node (node_5) at (45.000000bp,167.500000bp) [draw,draw=none] {$321$};
  \node (node_4) at (72.000000bp,114.500000bp) [draw,draw=none] {$231$};
  \node (node_3) at (18.000000bp,114.500000bp) [draw,draw=none] {$312$};
  \node (node_2) at (72.000000bp,61.500000bp) [draw,draw=none] {$213$};
  \node (node_1) at (18.000000bp,61.500000bp) [draw,draw=none] {$132$};
  \node (node_0) at (45.000000bp,8.500000bp) [draw,draw=none] {$123$};
  \draw [black,->] (node_1) ..controls (34.024000bp,77.634000bp) and (46.679000bp,89.586000bp)  .. (node_4);
  \draw [black,->] (node_4) ..controls (64.267000bp,130.110000bp) and (58.526000bp,140.950000bp)  .. (node_5);
  \draw [black,->] (node_0) ..controls (52.733000bp,24.106000bp) and (58.474000bp,34.950000bp)  .. (node_2);
  \draw [black,->] (node_2) ..controls (55.976000bp,77.634000bp) and (43.321000bp,89.586000bp)  .. (node_3);
  \draw [black,->] (node_1) ..controls (18.000000bp,76.805000bp) and (18.000000bp,87.034000bp)  .. (node_3);
  \draw [black,->] (node_3) ..controls (25.733000bp,130.110000bp) and (31.474000bp,140.950000bp)  .. (node_5);
  \draw [black,->] (node_0) ..controls (37.267000bp,24.106000bp) and (31.526000bp,34.950000bp)  .. (node_1);
  \draw [black,->] (node_2) ..controls (72.000000bp,76.805000bp) and (72.000000bp,87.034000bp)  .. (node_4);
\end{tikzpicture}
\end{center}
\caption{From left to right: the weak order on permutations of 3, the boolean partial order $TBool_3^{\rm Perm}$ on permutation TSSCPP of order $3$, and the strong Bruhat order on permutations of 3}
\label{fig:snorders}
\end{figure}
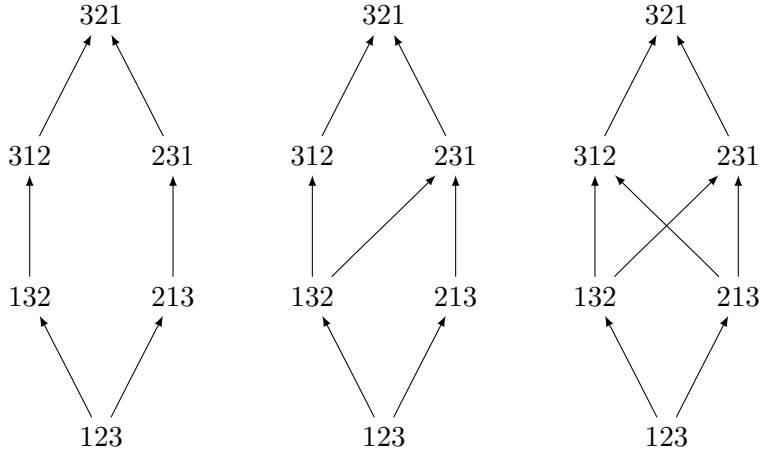

Finally, we note that though $T_n^{\rm Perm}$ and $TBool_n^{\rm Perm}$ differ, they contain the same Tamari and Catalan subposets. So we obtain the following, as a corollary of Theorems~\ref{thm:tamari} and \ref{thm:Catdistr}.

\begin{corollary}
\label{cor:tamcat2}
$TBool_n^{\rm Perm}$ contains both $Tam_n$ and $Cat_n$ as subposets. In particular, the subposet of $TBool_n^{\rm Perm}$ consisting of the $132$-avoiding permutations is isomorphic to $Tam_n$ and the subposet consisting of the $213$-avoiding permutations is isomorphic to $Cat_n$.
\end{corollary}

\begin{proof}
In Theorems~\ref{thm:tamari} and \ref{thm:Catdistr}, it was shown that the $T_n$ covering relations of Lemma~\ref{lem:magogboolcovers} on either the $132$- or $213$-avoiding subposets reduce to reverse componentwise comparison covering relations on the Catalan sequence $x_1,x_2,\ldots,x_n$, where $x_i$ was defined to equal $i$ plus the sum of row $n-i$ in the boolean triangle. This is equivalent to reverse componentwise comparison of the $132$- or $213$-avoiding permutation TSSCPP boolean triangles. Thus the $132$-avoiding subposet of $T_n^{\rm Perm}$ is equivalent to the $132$-avoiding subposet of $TBool_n^{\rm Perm}$, and likewise for the $213$-avoiding subposets.
\end{proof}

\subsection{Summary}
\label{subsec:summary}
In this section, we have discussed three posets: $A_n$ on ASM, and $T_n$ and $TBool_n$ on TSSCPP. $A_n$ and $T_n$ are distributive lattices and members of the tetrahedral poset family~\cite{STRIKERPOSET}, while $TBool_n$ is not a lattice in general (Proposition~\ref{prop:notlattice}). $A_n^{\rm Perm}$ and $TBool_n^{\rm Perm}$ are known permutation posets; $A_n^{\rm Perm}$ is isomorphic to the strong Bruhat order~\cite{TREILLIS} which is not a lattice, and we have shown $TBool_n^{\rm Perm}$ is isomorphic to $[2]\times[3]\times\cdots\times[n]$, which is a distributive lattice that sits between the weak and strong Bruhat orders (Corollary~\ref{cor:TBool}). While $T_n^{\rm Perm}$ is neither a well-studied permutation partial order nor a lattice (Proposition~\ref{prop:notlatticemagog}), it has Catalan subposets isomorphic to the Tamari lattice (Theorem~\ref{thm:tamari}) and the Catalan distributive lattice (Theorem~\ref{thm:Catdistr}). Thus, $T_n^{\rm Perm}$  may be an interesting new perspective from which to study permutations as well. 
We have also shown in Corollary~\ref{cor:tamcat2} that the same Tamari and Catalan subposets exist in $TBool_n^{\rm Perm}$ as in  $T_n^{\rm Perm}$. 
An additional point of contrast is that $A_n^{\rm Perm}$, as the strong Bruhat order, contains the Catalan distributive lattice as an induced subposet, but not the Tamari lattice. 

We hope that the study of these various ASM and TSSCPP partial orders will continue to provide insight on the combinatorics of these objects and the associated outstanding bijection problems. 

\section*{Acknowledgments}
The author thanks the anonymous referees for many helpful comments. The author also thanks Vic Reiner for suggesting Theorem \ref{thm:Catdistr} and for helpful comments on Section~\ref{subsec:tamari} and Dennis Stanton for helpful comments. The author also thanks the developers of Sage~\cite{sage} and the Sage--Combinat community~\cite{SageCombinat} for developing and sharing their code by which some of this research was conducted, including Propositions~\ref{prop:notlatticemagog} and \ref{prop:notlattice}, and by which we drew Figures~\ref{fig:A3T3}, \ref{fig:Cat_compare}, \ref{fig:Tpermmagog34}, \ref{fig:TBool3}, and \ref{fig:snorders}. She also would like to thank O.~Cheong, the developer of Ipe~\cite{ipe}, by which Figures~\ref{ex:funddomain3}, \ref{fig:bijex} (top), and \ref{fig:P5A5} were drawn. She thanks J.S.~Kim, who wrote the Ti\emph{k}Z code for plane partitions~\cite{PPTikz}, by which we drew Figures~\ref{fig:permbij} (left) and~\ref{ex:tikztsscpp}. 

\bibliographystyle{amsalpha}

\end{document}